\documentclass[11pt]{article}

\usepackage{booktabs}
\usepackage{graphicx,subfigure}
\usepackage{amsmath}
\usepackage{amssymb}
\usepackage{mathrsfs}
\usepackage{latexsym}
\usepackage{floatrow}
\usepackage{crop}
\usepackage{algorithmic,algorithm}
\usepackage{multirow}
\usepackage{bm}
\usepackage{bbm}
\usepackage{enumerate}
\usepackage{url}
\usepackage{array}
\usepackage{paralist}
\usepackage{diagbox}
\usepackage{booktabs}
\usepackage[dvipsnames]{xcolor}
\usepackage[colorlinks = true, pdfstartview = FitV, linkcolor = green, citecolor = green, urlcolor = green]{hyperref}
\usepackage{changepage}
\usepackage{transparent}

\usepackage{ifxetex}
\ifxetex
\usepackage{fontspec}
\else
\usepackage[utf8]{inputenc}
\usepackage[T1]{fontenc}
\usepackage{lmodern} 
\fi

\usepackage{rotating}

\usepackage[capitalise]{cleveref}
\crefname{equation}{}{}
\crefname{figure}{Figure}{Figures}
\creflabelformat{equation}{\textup{(#2#1#3)}}
\crefname{assumption}{Assumption}{Assumptions}
\crefname{condition}{Condition}{Conditions}

\usepackage{fullpage}
\usepackage{multirow}

\usepackage[sort,numbers]{natbib}

\usepackage{arydshln}
\usepackage{enumitem}
\setlist[enumerate]{leftmargin=*,wide=0em, noitemsep,nolistsep, label = {\bfseries \arabic*.}}
\setlist[itemize]{leftmargin=*,wide=0em, noitemsep,nolistsep}

\newif\iflongversion

\usepackage{pifont}
\usepackage{xspace}

\renewcommand\th{\textsuperscript{th}\xspace}

\usepackage{accents}

\usepackage{stackengine}
\stackMath
\newcommand\tsup[2][2]{%
	\def\useanchorwidth{T}%
	\ifnum#1>1%
	\stackon[-.5pt]{\tsup[\numexpr#1-1\relax]{#2}}{\scriptscriptstyle\sim}%
	\else%
	\stackon[.5pt]{#2}{\scriptscriptstyle\sim}%
	\fi%
}
\makeatletter
\newcommand{\longdash}[1][2em]{%
	\makebox[#1]{$\m@th\smash-\mkern-7mu\cleaders\hbox{$\mkern-2mu\smash-\mkern-2mu$}\hfill\mkern-7mu\smash-$}}
\makeatother
\newcommand{\omitskip}{\kern-\arraycolsep}

\newcommand{\df}{\mathrm{d}}
\newcommand{\real}{\mathbb{R}}

\DeclareMathOperator*{\argmin}{arg\,min}
\DeclareMathOperator*{\Argmin}{Arg\,min}

\newcommand{\sC}{\mathcal{C}}

\newcommand{\sX}{\mathcal{X}}

\renewcommand {\AA}  { {\mathbf{A}} }

\newcommand {\BB}  { {\mathbf{B}} }

\newcommand {\DD}  { {\mathbf{D}} }
\newcommand {\EE}  { {\mathbf{E}} }
\newcommand {\FF}  { {\mathbf{F}} }

\newcommand {\HH}  { {\mathbf{H}} }

\newcommand {\VV}  { {\mathbf{V}} }
\newcommand {\WW}  { {\mathbf{W}} }
\newcommand {\UU}  { {\mathbf{U}} }

\newcommand {\QQ}  { {\mathbf{Q}} }
\newcommand {\bSS}  { {\mathbf{S}} }
\newcommand {\JJ}  { {\mathbf{J}} }

\newcommand{\eye}{\mathbf{I}}

\renewcommand {\aa}  { {\bf a} }
\newcommand {\bb}  { {\bf b} }

\newcommand {\bgg}  { {\bf g} }

\newcommand {\yy}  { {\bf y} }

\newcommand {\rr}  { {\bf r} }
\newcommand {\uu}  { {\bf u} }
\newcommand {\qq}  { {\bf q} }
\newcommand {\pp}  { {\bf p} }

\newcommand {\vv}  { {\bf v} }
\newcommand {\ww}  { {\bf w} }
\newcommand {\xx}  { {\bf x} }
\newcommand {\zz}  { {\bf z} }

\newcommand {\zero}  { {\bf 0} }
\newcommand {\one}  { {\bf 1} }

\newcommand {\alphak}  { {{\alpha}_{k}} }

\newcommand {\UUk}  { {\UU_{k}} }

\newcommand {\HHk}  { {\HH_{k}} }

\newcommand {\bggk}  { {{\bgg}_{k}} }
\newcommand {\bggkk}  { {{\bgg}_{k+1}} }
\newcommand {\xxo}  { {{\xx}_{0}} }

\newcommand {\xxk}  { {{\xx}_{k}} }

\newcommand {\xxkk}  { {{\xx}_{k+1}} }
\newcommand {\xxs}  { {{\xx}^{\star}} }
\newcommand {\pps}  { {{\pp}^{\star}} }

\newcommand {\ppk}  { {{\pp}_{k}} }

\newcommand {\range}  { {\textnormal{Range}} }

\makeatletter
\newcommand*{\transpose}{%
	{\mathpalette\@transpose{}}%
}
\newcommand*{\@transpose}[2]{%
	\raisebox{\depth}{$\m@th#1\intercal$}%
}
\makeatother

\newcommand {\ppkt}  { {\pp_{k}^{(t)}} }

\newcommand{\hf}{\frac12}

\newcommand{\defeq}{\triangleq}

\definecolor{forestgreen}{rgb}{0.13, 0.55, 0.13}

\definecolor{amber}{rgb}{1.0, 0.75, 0.0}

\definecolor{bananayellow}{rgb}{.8, 0.6, 0}

\newcounter{comment}

\usepackage{amsthm}
\usepackage[framemethod=TikZ]{mdframed}

\newmdtheoremenv[%
linewidth = 1pt,%
roundcorner = 10pt,%
leftmargin = 0,%
rightmargin = 0,%
backgroundcolor = green!3,%
outerlinecolor = blue!70!black,%
splittopskip = \topskip,%
ntheorem = true,%
]{theorem}{Theorem}

\newmdtheoremenv[%
linewidth = 1pt,%
roundcorner = 10pt,%
leftmargin = 0,%
rightmargin = 0,%
backgroundcolor = green!3,%
outerlinecolor = blue!70!black,%
splittopskip = \topskip,%
ntheorem = true,%
]{corollary}{Corollary}

\newmdtheoremenv[%
linewidth = 1pt,%
roundcorner = 10pt,%
leftmargin = 0,%
rightmargin = 0,%
backgroundcolor = green!3,%
outerlinecolor = blue!70!black,%
splittopskip = \topskip,%
ntheorem = true,%
]{lemma}{Lemma}

\newmdtheoremenv[%
linewidth = 1pt,%
roundcorner = 10pt,%
leftmargin = 0,%
rightmargin = 0,%
backgroundcolor = blue!3,%
outerlinecolor = blue!70!black,%
splittopskip = \topskip,%
ntheorem = true,%
]{definition}{Definition}

\newmdtheoremenv[%
linewidth = 1pt,%
roundcorner = 10pt,%
leftmargin = 0,%
rightmargin = 0,%
backgroundcolor = green!3,%
outerlinecolor = blue!70!black,%
splittopskip = \topskip,%
ntheorem = true,%
]{proposition}{Proposition}

\newmdtheoremenv[%
linewidth = 1pt,%
roundcorner = 10pt,%
leftmargin = 0,%
rightmargin = 0,%
backgroundcolor = green!3,%
outerlinecolor = blue!70!black,%
splittopskip = \topskip,%
ntheorem = true,%
]{condition}{Condition}

\newmdtheoremenv[%
linewidth = 1pt,%
roundcorner = 10pt,%
leftmargin = 0,%
rightmargin = 0,%
backgroundcolor = yellow!3,%
outerlinecolor = blue!70!black,%
splittopskip = \topskip,%
ntheorem = true,%
]{assumption}{Assumption}

\theoremstyle{definition}
\newmdtheoremenv[%
linewidth = 1pt,%
roundcorner = 10pt,%
leftmargin = 0,%
rightmargin = 0,%
backgroundcolor = cyan!3,%
outerlinecolor = blue!70!black,%
splittopskip = \topskip,%
ntheorem = true,%
]{example}{Example}

\theoremstyle{definition}
\newmdtheoremenv[%
linewidth = 1pt,%
roundcorner = 10pt,%
leftmargin = 0,%
rightmargin = 0,%
backgroundcolor = red!3,%
outerlinecolor = blue!70!black,%
splittopskip = \topskip,%
ntheorem = true,%
]{remark}{Remark}

\mdtheorem[%
linewidth = 1pt,%
roundcorner = 10pt,%
leftmargin = 0,%
rightmargin = 0,%
backgroundcolor = yellow!3,%
outerlinecolor = blue!70!black,%
splittopskip = \topskip,%
ntheorem = true,%
]{informal}{Definition}

\usepackage{tikz}
\usepackage{xparse}

\NewDocumentCommand\DownArrow{O{2.0ex} O{black}}{%
	\mathrel{\tikz[baseline] \draw [<-, line width=0.5pt, #2] (0,0) -- ++(0,#1);}
}

\usepackage{listings} 

\definecolor{mygreen}{rgb}{0,0.6,0}
\definecolor{mygray}{rgb}{0.5,0.5,0.5}
\definecolor{mymauve}{rgb}{0.58,0,0.82}
\definecolor{codegreen}{rgb}{0,0.6,0}
\definecolor{codegray}{rgb}{0.5,0.5,0.5}
\definecolor{codepurple}{rgb}{0.58,0,0.82}
\definecolor{backcolour}{rgb}{0.95,0.95,0.92}

\lstdefinestyle{mystyle}{
	backgroundcolor=\color{backcolour},   
	commentstyle=\color{codegreen},
	keywordstyle=\color{magenta},
	numberstyle=\tiny\color{codegray},
	stringstyle=\color{codepurple},
	basicstyle=\ttfamily\footnotesize,
	breakatwhitespace=false,         
	breaklines=true,                 
	captionpos=b,                    
	keepspaces=true,                 
	numbers=left,                    
	numbersep=5pt,                  
	showspaces=false,                
	showstringspaces=false,
	showtabs=false,                  
	tabsize=2
}
\newcommand*\dotprod[1]{\left\langle #1\right\rangle}
\newcommand*\vnorm[1]{\left\| #1\right\|}

\newcommand*\bigO[1]{\mathcal O\left( #1\right)}

\newcommand\bbR{\ensuremath{\mathbb{R}}} 

\usepackage{dsfont}
\newcommand*\indicfun[1]{\mathds{1}_{\left\{#1\right\}}}

\begin{document}
	
	\title{Newton-MR: Inexact Newton Method With Minimum Residual Sub-problem Solver}
	
	\author{Fred Roosta\thanks{School of Mathematics and Physics, University of Queensland, Australia, The Centre for Information Resilience (CIRES), Australia, and International Computer Science Institute, Berkeley, USA. Email:  \tt{fred.roosta@uq.edu.au}} 
		\and 
		Yang Liu\thanks{School of Mathematics and Physics, University of Queensland, Australia. Email: \tt{yang.liu2@uq.edu.au}}
		\and
		Peng Xu\thanks{Institute for Computational and Mathematical Engineering, Stanford University, USA. Email: \tt{pengxu@stanford.edu}}
		\and
		Michael W. Mahoney\thanks{International Computer Science Institute and Department of Statistics, University of California at Berkeley, USA. Email: \tt{mmahoney@stat.berkeley.edu}}
	}
	
	\date{\today}
	\maketitle

\abstract{
We consider a variant of inexact Newton Method \cite{eisenstat1994globally,brown1994convergence}, called Newton-MR, in which the least-squares sub-problems are solved approximately using Minimum Residual method \cite{paige1975solution}. 
By construction, Newton-MR can be readily applied for unconstrained optimization of a class of non-convex problems known as invex, which subsumes convexity as a sub-class. 
For invex optimization, instead of the classical Lipschitz continuity assumptions on gradient and Hessian, Newton-MR's global convergence can be guaranteed under a weaker notion of joint regularity of Hessian and gradient. 
We also obtain Newton-MR's problem-independent local convergence to the set of minima. 
We show that fast local/global convergence can be guaranteed under a novel inexactness condition, which, to our knowledge, is much weaker than the prior related works. 
Numerical results demonstrate the performance of Newton-MR as compared with several other Newton-type alternatives on a few machine learning problems.
}

\section{Introduction}
\label{sec:intro}
For a twice differentiable function $f:\mathbb{R}^{d} \rightarrow \mathbb{R}$, consider the unconstrained optimization problem
\begin{align}
	\label{eq:obj}
	\min_{\xx \in \mathbb{R}^{d}} f(\xx).
\end{align}
The canonical example of second-order methods is arguably the classical Newton's method whose iterations are typically written as 
\begin{align}
	\label{eq:newton_iterations}
	\xxkk = \xxk + \alphak \ppk, \quad \text{ with } \quad \ppk = -\left[\HH(\xxk)\right]^{-1} \bgg(\xxk),
\end{align}
where $ \xxk$, $\bgg(\xxk) = \nabla f(\xxk)$, $\HH(\xxk) = \nabla^{2} f(\xxk)$, and $\alphak $ are respectively the current iterate, the gradient,  the Hessian matrix, and the step-size that is often chosen using an Armijo-type line-search to enforce sufficient decrease in $ f $ \cite{nocedal2006numerical}. When $ f $ is smooth and strongly convex, it is well known that the local and global convergence rates of the classical Newton's method are, respectively, quadratic and linear \cite{nesterov2004introductory,boyd2004convex,nocedal2006numerical}. 
For such problems, the Hessian matrix is uniformly positive definite. As a result, if forming the Hessian matrix explicitly and/or solving the linear system in \cref{eq:newton_iterations} exactly is prohibitive, the update direction $ \ppk $ is obtained approximately using conjugate gradient (CG), resulting in the celebrated Newton-CG \cite{nocedal2006numerical}. 

Newton-CG is elegantly simple in that its iterations merely involve solving linear systems followed by a certain type of line-search. Despite its simplicity, it has been shown to enjoy various desirable theoretical and algorithmic properties including insensitivity to problem ill-conditioning~\cite{ssn2018,xu2016sub} as well as robustness to hyper-parameter tuning~\cite{kylasa2018gpu,berahas2017investigation}. However, in the absence of strong-convexity, Newton's method lacks any favorable convergence analysis. Indeed, recall that to obtain the update direction, $ \ppk $, Newton-CG aims at (approximately) solving quadratic sub-problems of the form (\cite{nocedal2006numerical})
\begin{align}
	\label{eq:quadratic}
	\min_{\pp \in \real^{d}} \dotprod{\bgg(\xxk), \pp}  + \hf \dotprod{ \pp, \HH(\xxk) \pp}.
\end{align}
However, if the Hessian is indefinite (as in non-convex settings) or if its positive semi-definite but $\bgg(\xxk) \notin \text{Range}\left(\HH(\xxk)\right)$ (as in weakly convex problems), \cref{eq:quadratic} is simply unbounded below. 

In this light, many Newton-type variants have been proposed which aim at extending Newton-CG beyond strongly convex problems to more general settings, e.g., trust-region \cite{conn2000trust} and cubic regularization \cite{nesterov2006cubic,cartis2011adaptiveI,cartis2011adaptiveII}. However, many of these extensions involve sub-problems that are themselves non-trivial to solve, e.g., the sub-problems of trust-region and cubic regularization methods are themselves non-linear and (potentially) non-convex \cite{xuNonconvexEmpirical2017,xuNonconvexTheoretical2017}. 
This has prompted many authors to design effective methods for approximately solving such sub-problems, e.g., \cite{steihaug1983conjugate,curtis2021trust,gould1999solving, lenders2018trlib,cartis2011adaptiveI,carmon2016gradient,bianconcini2015use}. These methods can be highly effective in many settings. For example, the trust-region method coupled with CG-Steihaug \cite{steihaug1983conjugate} has shown great promise in several machine learning applications \cite{xuNonconvexEmpirical2017}. However, solving such non-trivial sub-problems can be a major challenge in many other settings. For instance, despite the obvious advantages of the CG-Steihaug method, e.g., one matrix-vector product per iteration and the ability to naturally detect negative curvature directions, it is not guaranteed to solve the trust-region sub-problem to an arbitrary accuracy. This can have negative consequences for the convergence speed of the trust-region method. Indeed, if either of the negative curvature or the boundary is encountered too early, the CG-Steihaug method terminates and the resulting step is only slightly, if at all, better than the Cauchy direction \cite{conn2000trust}. If this occurs too often, the trust-region method's convergence can slow down to be comparable to that of the simple gradient descent method.

Recently, \cite{mishchenko2021regularized} proposes a novel Newton-type method, which not only enjoys sub-problems in the form of linear-systems, but as long as these sub-problems are solved exactly, it has provably fast global convergence for weakly convex objectives. 
To allow for optimization of a larger class of problems beyond convex, while maintaining the simplicity offered by linear sub-problems, one can consider replacing \cref{eq:quadratic} with an equivalently simple ordinary least-squares (OLS)
\begin{align}
	\label{eq:ols}
	\min_{\pp \in \real^{d}} \; \hf \vnorm{\HH(\xxk) \pp + \bgg(\xxk)}^{2}.
\end{align}
Clearly, regardless of the Hessian matrix, there is always a solution (in fact at times infinitely many solutions) to the the above OLS problem. For example, exact minimum norm solution to the above OLS amounts to iterations of the form
\begin{align}
	\label{eq:newton_mr_iterations}
	\xxkk = \xxk + \alphak \ppk, \quad \text{ with } \quad \ppk = -\left[\HH(\xxk)\right]^{\dagger} \bgg(\xxk),
\end{align}   
where $ \AA^{\dagger} $ denotes the Moore-Penrose generalized inverse of matrix $ \AA $, and $ \alphak $ is an appropriately chosen step-size. 
Our aim in this paper is to give an in-depth treatment of the above alternative and study the convergence properties of the resulting algorithm, which we call \emph{Newton-MR}\footnote{The term ``MR'' refers to the fact that the sub-problems are in the form of \emph{\textbf{M}inimum \textbf{R}esidual}, i.e., least squares, as depicted in \cref{alg:NewtonMR_Ex,alg:NewtonMR}.}, under a variety of settings.

The rest of this paper is organized as follows. We end \cref{sec:intro} by introducing the notation used throughout the paper. 
In \cref{sec:newton_mr}, we take a look at high-level consequences of using a direction from (an approximation of) \cref{eq:ols}.
Convergence properties of Newton-MR are gathered in \cref{sec:analysis}. This is done, first, by extensive discussion on various assumptions underlying the analysis in \cref{sec:assumption}, followed by a detailed theoretical development to obtain local and global convergence of Newton-MR in \cref{sec:convergence}. Numerical examples are provided in \cref{sec:experiments}. Conclusions and further thoughts are gathered in \cref{sec:conclusion}.

\paragraph{Notation.}
In what follows, vectors and matrices are denoted by bold lowercase and bold uppercase letters, e.g., $\vv$ and $\VV$, respectively. We use regular lowercase and uppercase letters to denote scalar constants, e.g., $ c $  or $ L $. 
The transpose of a real vector $\vv$ is denoted by $ \vv^{\intercal} $. 
For two vectors, $ \vv,\ww $, their inner-product is denoted as $ \dotprod{\vv, \ww}  = \vv^{\intercal} \ww$. For a vector $\vv$, and a matrix $\VV$, $\|\vv\|$ and $\|\VV\|$ denote the vector $\ell_{2}$ norm and the matrix spectral norm, respectively. 
For any $ \xx,\zz \in \mathbb{R}^{d}$, $ \yy \in \left[\xx, \zz\right] $ denotes $ \yy = \xx + \tau (\zz - \xx) $ for some $ 0 \leq \tau \leq 1 $. 
For a natural number $ n $, we denote  $ [n] = \left\{ 1,2,\ldots,n \right\} $. 
For any finite collection $ \mathcal{S} $, its cardinality is denoted by $ |\mathcal{S}| $. 
The subscript, e.g., $\xxk$, denotes iteration counter. 
For a matrix $ \AA \in \real^{d \times d} $, $ \text{Range}\left(\AA\right) $ and $ \text{Null}\left(\AA\right) $ denote, respectively, its range, i.e., $ \text{Range}\left(\AA\right) = \left\{\AA \xx \mid \xx \in \real^{d} \right\}$ and its null-space, i.e., $ \text{Null}\left(\AA\right) = \left\{\yy \in \real^{d} \mid \AA \yy  = 0 \right\}$. 
The Moore-Penrose generalized inverse of a matrix $ \AA $ is denoted by $ \AA^{\dagger} $. 
The identity matrix is written as $\eye$ and the indicator function of a set $ \mathcal{S} $ is denoted as $ \indicfun{\mathcal{S}} $. 
When the minimum is attained at more than one point, ``$ \Argmin $'' denotes the set of minimizers, otherwise ``$ \argmin $'' implies a unique minimizer. 
Finally, we use $\bgg(\xx) \triangleq \nabla f\left(\xx\right)$ and $\HH(\xx) \triangleq \nabla^{2} f\left(\xx\right) $ for the gradient and the Hessian of $f$ at $\xx$, respectively, and at times we drop the dependence on $ \xx $ by simply using $ \bgg $ and $ \HH $, e.g., $ \bggk = \bgg(\xxk) $ and $ \HHk = \HH(\xxk) $.

\section{Newton-MR}
\label{sec:newton_mr}
Before providing a detailed look into the convergence implications of using \cref{eq:ols}, let us first study higher-level consequences of such a choice as they relate to solving \cref{eq:obj}.

\subsection{Minimum Residual Sub-problem Solver}
\label{sec:minres-qlp}
For solving OLS problems, a plethora of iterative solvers exists.
Each of these iterative methods can be highly effective once applied to the class of problems for which they have been designed.
Among them, the one method which has been specifically designed to effectively obtain the least-norm solution for all systems, compatible or otherwise, involving real symmetric but potentially indefinite/singular matrices (such as those in \cref{eq:ols}) is MINRES-QLP \cite{choi2014algorithm, choi2011minres}. In fact, by requiring only marginally more computations/storage, it has been shown that MINRES-QLP is numerically more stable than MINRES on ill-conditioned problems.

Perhaps a natural question at this stage is ``\emph{what are the descent properties of the (approximate) solution to \cref{eq:ols} using MINRES-QLP?}'' Recall that the $ t\th $ iteration of MINRES-QLP yields a search direction, $ \pp_{k}^{(t)} $, which satisfies
\begin{align}
	\label{eq:minres_sub}
	\pp_{k}^{(t)} = \argmin_{\pp \in \mathcal{K}_{t}(\HHk, \bggk)}  r_{k}(\pp) \triangleq \vnorm{\bggk + \HHk \pp},
\end{align}
where $ \mathcal{K}_{t}(\HHk, \bggk) = \text{Span}\{\bggk,\HHk \bggk, \ldots, \left[\HHk\right]^{t-1} \bggk \}$ denotes the Krylov sub-space of order $ t $, \cite{golub2013matrix}.
Since $ \bm{0} \in \mathcal{K}_{t}(\HHk, \bggk)$, it follows that for any $ t $ in \cref{eq:minres_sub}, we have $ r_{k}(\pp_{k}^{(t)}) \leq r_{k}(\bm{0}) $, which implies $$\dotprod{\pp_{k}^{(t)},\HHk \bggk} \leq - \vnorm{\HHk \pp_{k}^{(t)}}^{2}/2 \leq 0.$$ 
In other words, $ \pp_{k}^{(t)} $ from MINRES, for all $ t $, is a descent direction for $ \vnorm{\bgg(\xx)} $ at $ \xxk $. 
Let $ \ppk = \ppkt $ for some $ t \geq 1 $. As a consequence of such particular descent property of $ \ppk $, the step-size, $ \alphak $, can be chosen using Armijo-type line search \cite{nocedal2006numerical} to reduce the norm of the gradient. Hence, we chose $\alphak \leq 1$ as the largest value such that for some $0< \rho < 1$, we get
\begin{align}
	\label{eq:armijo_gen}
	\vnorm{{\bggkk}}^{2} &\leq \vnorm{\bggk}^{2} + 2 \rho \alphak \dotprod{\ppk, \HHk \bggk}.
\end{align}
This can be approximately achieved using any standard back-tracking line-search strategy. 

\subsection{Invexity}
\label{sec:invexity}
From the above discussion, we always have $ \dotprod{\ppk, \HHk \bggk} \leq 0 $, which implies that 
\begin{align}
	\label{eq:grad_descent}
	\vnorm{{\bggkk}} &\leq \vnorm{\bggk}.
\end{align}
In other words, by appropriate application of Hessian and regardless of non-convexity of the problem, one can always obtain descent directions corresponding to the auxiliary problem 
\begin{align}
	\label{eq:obj_grad}
	\min_{\xx \in \mathbb{R}^{d}} \vnorm{\bgg(\xx)}.
\end{align}
In certain applications, e.g., chemical physics \cite{mciver1972structure,angelani2000saddles} and deep learning \cite{frye2019numerically,frye2021critical,frye2019critical}, instead of minimizing the function, the goal is to find zeros of its gradient field, which can obtained by solving \cref{eq:obj_grad}. However, when the goal is solving \cref{eq:obj}, a natural question is ``\emph{what is the class of objective functions for which \cref{eq:obj_grad} is equivalent to \cref{eq:obj}?}''
Clearly, any global optimum of \cref{eq:obj} is also a solution to \cref{eq:obj_grad}. However, the converse only holds for a special class of non-convex functions, known as \emph{invex} \cite{mishra2008invexity,hanson1981sufficiency}. 
\begin{definition}[Invexity]
	\label{def:invex}
	Let $\mathcal{X} \subseteq \real^d$ be an open set. A differentiable function $f: \mathcal{X} \rightarrow \bbR$ is said to be invex on $ \mathcal{X} $ if there exists a vector-valued function $\bm{\phi}: \mathcal{X} \times \mathcal{X} \rightarrow \real^d$ such that
	\begin{align}
		\label{eq:invex}
		f(\yy) - f(\xx) \geq \dotprod{\bm{\phi}(\yy, \xx),\bgg(\xx)}, \hspace{1cm} \forall \xx, \yy \in \mathcal{X}.
	\end{align}
	The class of functions satisfying \cref{eq:invex} with the same $ \bm{\phi} $ is denoted by $ \mathcal{F}_{\bm{\phi}} $. The class of all invex functions is denoted by $ \mathcal{F} $.
\end{definition}
Invexity characterizes the class of functions for which the first order optimality condition is also sufficient.  It is also easily seen that, by choosing $\bm{\phi}(\yy,\xx) = \yy-\xx$ in \cref{def:invex}, differentiable convex functions are subsumed in the class of invex functions; see  \cite{mishra2008invexity,ben1986invexity} for detailed treatments of invexity. 
%
Much like all generalized convex functions \cite{cambini2008generalized}, invexity has been introduced in order to weaken, as much as possible, the convexity requirements in optimization problems, and extend many results related to convex optimization theory to more general functions. 

Invexity has recently garnered attention in the machine learning community. Restricted to a special subclass, often referred to as Polyak-\L{}ojasiewicz functions (cf.\ \cref{sec:gpl}), numerous works have established the convergence of first-order optimization algorithms, e.g.,  \cite{karimi2016linear,bassily2018exponential,oymak2019overparameterized,dereich2021convergence,kim2021convergence,vaswani2019fast}. In addition, it has recently been shown that several deep learning models enjoy such invex properties in certain regimes, e.g., \cite{liu2022loss,barboni2021global}. For more general non-convex machine learning models, \cite{crane2021invexifying} proposes a novel regularization framework, which ensures the invexity of the regularized function and is shown to be advantageous over the classical $ \ell_{2} $-regularization in terms of the sensitivity to hyper-parameter tuning and out-of-sample generalization performance.

\subsection{Connections to Root Finding Algorithms}
\label{sec:root}
Considering \cref{eq:obj_grad} in lieu of $ \bgg(\xx) = 0 $ is, in fact, a special case of a more general framework for solving non-linear system of equations involving a vector valued function $ \mathbf{F}: \real^{m} \rightarrow \real^{n}$. More specifically, as an alternative to solving $ \mathbf{F}(\xx) = 0 $, minimization of $ \|\mathbf{F}(\xx)\| $ has been considered extensively in the literature; e.g.,  \cite{polyak2017solving,chen1994newton,nesterov2007modified,yuan2011recent,bellavia2010convergence,zhao2016global,bellavia2018levenberg}. Consequently, Newton-MR can be regarded as a member of the class of inexact Newton methods with line-search for solving nonlinear systems \cite{eisenstat1994globally,eisenstat1996choosing,morini1999convergence,eisenstat1994globally,dembo1982inexact}. In our case, the non-linear system of equations arises from optimality condition $ \bgg(\xx) = 0 $, i.e., $ \mathbf{F}(\xx) = \bgg(\xx) $, which has a symmetric Jacobian. Hence, through the perspective of the application of MINRES within its iterations, Newton-MR can be viewed as a special case of more general Newton-GMRES algorithms for solving non-linear systems \cite{brown1994convergence,brown1990hybrid,bellavia2000hybrid,bellavia2001globally,an2007globally,kelley1995iterative}. 

In light of these connections, the bare-bones iterations of Newton-MR do not constitute novel elements of our work here. 
What sets our contributions in this paper apart is studying the plethora of desirable theoretical and algorithmic properties of Newton-MR in the context of \cref{eq:obj_grad} and its connection to \cref{eq:obj}. In particular, we give a thorough treatment of the implications of the assumptions that we make as well as the novel inexactness conditions that we propose as they relate to convergence of Newton-MR. For example, in most prior works on Newton-GMRES, the Jacobian of the non-linear function is assumed full-rank \cite{brown1994convergence,bellavia2001globally,bellavia2000hybrid,walker1990least}. This, in our setting here, amounts to assuming that the Hessian matrix is invertible, which can be easily violated in the non-convex settings. In fact, to address rank deficient cases, the ubiquitous approach has been to employ regularization techniques such as Levenberg-Marquardt \cite{fan2005quadratic,behling2019local,behling2012unified,dennis1996numerical,gratton2007approximate,yamashita2001rate,gratton2007approximate,dennis1996numerical}. By introducing a novel assumption on the interplay between the gradient and the sub-space spanned by the Hessian matrix, we obviate the need for full-rank assumptions or the Hessian regularization techniques. As another example, to guarantee convergence, we will introduce a novel inexactness condition that constitute a major relaxation from the typical relative residual conditions used, almost ubiquitously, in prior works.

\section{Theoretical Analysis}
\label{sec:analysis}
In this section, we study the convergence properties of some variants of Newton-MR (\cref{alg:NewtonMR_Ex,alg:NewtonMR}). For this, in \cref{sec:assumption}, we first give the assumptions underlying our analysis. 
Under these assumptions, in \cref{sec:convergence}, we  provide detailed local/global convergence analysis of these Newton-MR variants. 

\subsection{Assumptions}
\label{sec:assumption}
For the theoretical analysis of Newton-MR, we make the following blanket assumptions regarding the properties of the objective function $ f $ in \cref{eq:obj}. Some of these assumptions might seem unconventional at first, however, they are simply generalizations of many typical assumptions made in the similar literature. For example, a strongly-convex function with Lipschitz continuous gradient and Hessian satisfies all of the following assumptions in this Section. 

\begin{assumption}[Differentiability]
	\label{assmpt:diff}
	The function $ f $ is twice-differentiable. 
\end{assumption}
In particular, all the first partial derivatives are themselves differentiable, but the second partial derivatives are allowed to be discontinuous. Recall that requiring the first partials be differentiable implies the equality of crossed-partials, which amounts to the symmetric Hessian matrix \cite[pp.\ 732-733]{hubbard2015vector}. 

In the literature for the analysis of non-convex Newton-type methods for \cref{eq:obj}, to obtain \emph{non-asymptotic quantitative} convergence rates, it is typically assumed that the function is sufficiently smooth in that its gradient and Hessian are Lipschitz continuous. Specifically, for some $ 0 \leq L_{\bgg} < \infty$, $0\leq L_{\HH} < \infty $, and $ \forall ~\xx, \yy \in \real^{d} $, it is assumed that
\begin{subequations}
	\label{eq:lip_usual}
	\begin{align}
		\vnorm{\bgg(\xx) - \bgg(\yy)} &\leq L_{\bgg} \vnorm{\xx - \yy}, 	\label{eq:lip_usual_grad} \\
		\vnorm{\HH(\xx) - \HH(\yy)} &\leq L_{\HH} \vnorm{\xx - \yy}. \label{eq:lip_usual_hessian}
	\end{align}
\end{subequations}
For Newton-MR, the smoothness is required for the auxiliary function $\vnorm{\bgg(\xx)}^{2}$, which amounts to a more relaxed smoothness condition.
 
\begin{assumption}[Moral-smoothness]
	\label{assmpt:lipschitz_special}
	For any $ \xx_{0} \in \real^{d} $, there is $ 0 \leq L(\xx_{0}) < \infty $, such that 
	\begin{align}
		\label{eq:lip_special}
		\vnorm{\HH(\xx) \bgg(\xx)-\HH(\yy) \bgg(\yy)} \leq L(\xx_{0}) \vnorm{\xx - \yy}^{\beta}, \quad \forall (\xx,\yy) \in \mathcal{X}_{0} \times \real^{d}, 
	\end{align}
	where $\mathcal{X}_{0} \triangleq \left\{\xx \in \real^{d} \mid \vnorm{\bgg(\xx)} \leq \vnorm{\bgg(\xx_{0})}\right\}$ and $0 < \beta < \infty$.
\end{assumption}
Note that in \cref{assmpt:lipschitz_special}, the constant $ L(\xx_{0}) $ depends on the choice of $ \xx_{0} $. 
We recall that, for $ \beta = 1 $, \cref{assmpt:lipschitz_special} is similar to the assumption that is often made in the analysis of many non-linear least-squares algorithms for minimizing $\vnorm{\FF(\xx)}^{2}$ for some non-linear mapping $ \FF(\xx) $ in which  $\nabla \vnorm{\FF(\xx)}^{2}$ is assumed Lipschitz continuous (on a level-set).
%
%
We refer to \cref{assmpt:lipschitz_special} as ``moral-smoothness'' since by \cref{eq:lip_special}, it is only the action of Hessian on the gradient that is required to be Lipschitz continuous, and each gradient and/or Hessian individually can be highly irregular, e.g., gradient can be very non-smooth and Hessian can even be discontinuous.

\begin{example}[A Morally-smooth Function with Discontinuous Hessian]
	\label{example:morall_smooth_disc_Hessian}
	Consider a quadratically smoothed variant of hinge-loss function, 
	\begin{align*}
		f(\xx) = \hf \max \Big\{ 0, b\dotprod{\aa,\xx}\Big\}^{2},
	\end{align*}
	for a given $ (\aa,b) \in \real^{d} \times \real $. It is easy to see that 
	\begin{align*}
		\bgg(\xx) &= b^2 \dotprod{\aa,\xx} \aa\; \indicfun{b\dotprod{\aa,\xx} > 0}, \; \quad \forall \xx \in \real^{d},\\
		\HH(\xx) &= b^{2} \aa \aa^{\intercal} \; \indicfun{b\dotprod{\aa,\xx} > 0}, \quad \quad \;\;\; \forall \xx \notin \mathcal{N},
	\end{align*}
	where $\mathcal{N} \triangleq \left\{ \xx \in \real^{d} \mid  b \dotprod{\aa,\xx} = 0 \right\}$. Clearly, Hessian is discontinuous on $ \mathcal{N} $. Since $ \mathcal{N} $ is a set of measure zero, i.e., $ \mu(\mathcal{N}) = 0 $ with respect to the Lebesgue measure $ \mu $, we can arbitrarily define $ \HH(\xx) \triangleq \bm{0} $ for $ \xx \in \mathcal{N} $ (in fact, any other definition would work as well). It follows that for any $ \xx,\yy \in \real^{d} $, we have
	\begin{align*}
		\vnorm{\HH(\xx)\bgg(\xx) - \HH(\xx)\bgg(\xx)} &= b^{4} \vnorm{\aa}^{2} \vnorm{\dotprod{\aa,\xx} \aa  \; \indicfun{b\dotprod{\aa,\xx} > 0} - \dotprod{\aa,\yy} \aa  \; \indicfun{b\dotprod{\aa,\yy} > 0}  } \\
		&\leq b^{4} \vnorm{\aa}^{3} \left|\dotprod{\aa,\xx} \; \indicfun{b\dotprod{\aa,\xx} > 0} - \dotprod{\aa,\yy} \; \indicfun{b\dotprod{\aa,\yy} > 0}\right| \\
		&\leq b^{4} \vnorm{\aa}^{4} \vnorm{\xx - \yy}.
	\end{align*}
	The last inequality follows by considering four cases with $ \indicfun{b\dotprod{\aa,\xx} > 0} = \indicfun{b\dotprod{\aa,\yy} > 0} $ and $ \indicfun{b\dotprod{\aa,\xx} > 0} \neq \indicfun{b\dotprod{\aa,\yy} > 0} $. Indeed, for the former two cases, the last inequality follows immediately. For the latter two cases, suppose without loss of generality that $ \indicfun{b\dotprod{\aa,\xx} > 0} = 1 $ and $ \indicfun{b\dotprod{\aa,\yy} > 0} = 0 $. In this case, we use the fact that by $ b \dotprod{\aa,\yy} \leq 0 $, we have $ b \dotprod{\aa,\xx} \leq b \dotprod{\aa,\xx} - b \dotprod{\aa,\yy}$.
\end{example}

It is easy to show that \cref{eq:lip_special} is implied by \cref{eq:lip_usual}, and hence, it constitutes a relaxation on the typical smoothness requirements for gradient and Hessian, commonly found in the literature. 

\begin{lemma}[Moral-smoothness \cref{eq:lip_special} is less strict than smoothness  \cref{eq:lip_usual}]
	\label{lemma:lipschitz_special}
	Suppose \cref{eq:lip_usual} holds. Then, for any $ \xx_{0} $, we have \cref{eq:lip_special} with $ \beta = 1$ and  $L(\xx_{0}) = L_{\bgg}^{2} + L_{\HH} \vnorm{\bgg(\xx_{0})}$.
\end{lemma}
\begin{proof}
	Suppose \cref{eq:lip_usual} is satisfied and fix any $ \xx_{0} $ with the corresponding $ \mathcal{X}_{0} $ as defined in \cref{assmpt:lipschitz_special}. Since for any two matrices $ \AA,\BB $ and two vectors $ \xx,\yy $, we have $ \|\AA\xx - \BB \yy\| \leq \|\AA\|\|\xx-\yy\|+\|\yy\| \|\AA-\BB\| $, it implies that $\forall \xx \in \mathcal{X}_{0}, ~\forall \yy \in \real^{d}$
	\begin{align*}
		\vnorm{\HH(\yy) \bgg(\yy) - \HH(\xx) \bgg(\xx)} & \leq \vnorm{\HH(\yy)} \vnorm{\bgg(\xx)  - \bgg(\yy)} + \vnorm{\bgg(\xx)} \vnorm{\HH(\xx) - \HH(\yy)} \\
		& \leq L_{\bgg}^{2} \vnorm{\yy  - \xx} + L_{\HH} \vnorm{\bgg(\xx)} \vnorm{\yy  - \xx} \leq \left(L_{\bgg}^{2} + L_{\HH} \vnorm{\bgg(\xx_{0})}\right) \vnorm{\yy  - \xx},
	\end{align*}
	where the last inequality follows since $ \xx \in \mathcal{X}_{0} $. 
\end{proof}
%

By \cref{lemma:lipschitz_special}, any smooth function satisfying \cref{eq:lip_usual} would also satisfy \cref{assmpt:lipschitz_special} with $ \beta = 1 $. Below, we bring examples which show that the converse does not necessarily hold. 

\begin{example}[Smoothness \cref{eq:lip_usual} is strictly stronger than moral-smoothness \cref{eq:lip_special}]
	In this example, for three regimes of $ 0 < \beta < 1, \beta = 1 $, and $ \beta > 1 $, we give concrete counter-examples of functions on $ \mathbb{R} $, i.e., $ d = 1 $, which satisfy \cref{eq:lip_special} but not \cref{eq:lip_usual}.
	\label{example:morally_smooth}
	\begin{enumerate}
		\item For $0< \beta < 1$, consider a solution to the following ODE 
		\begin{align*}
			f''(x) f'(x) &= x^{\beta}, \; x > x_{0} \triangleq \Big((1+\beta)a\Big)^{\frac{1}{(1+\beta)}}, \quad f(x_{0}) = b, \quad f'(x_{0}) = 0,
		\end{align*}
		where $ b \in \mathbb{R}, a \in \mathbb{R}^{+}$. Such a solution can be of the form
		\begin{subequations}
			\begin{align}
				\label{eq:example_f_beta_leq_1}
				f(x) = b + \sqrt{\frac{2}{1+\beta}}  \int_{x_{0}}^{x} \sqrt{t^{\beta + 1} - (1+\beta)a} \; d t,
			\end{align}
			and satisfies \cref{eq:lip_special} with $0 < \beta < 1$, and $L=1$. It is easily verified that $ f''(x) $ and $ f'''(x) $ are both unbounded; hence $ f $ violates \cref{eq:lip_usual}. Note that, on $ (x_{0}, \infty) $, \cref{eq:example_f_beta_leq_1} is convex, and hence by definition, it is invex.
			
			\item For $\beta = 1$, the condition \cref{eq:lip_special} implies that $\vnorm{\left[\nabla^{2}f(\xx)\right]^{2} + \dotprod{\nabla^{3}f(\xx),\bgg(\xx)}} \leq L$.
			In one variable, for $ b \in \mathbb{R}, a \in \mathbb{R}^{++} $, the solution to the following ODE 
			\begin{align*}
				f''(x) f'(x) = x, \quad x > x_{0} \triangleq \sqrt{2a}, \quad f(x_{0}) = b - a \log(\sqrt{2a}), \quad f'(x_{0}) = 0,
			\end{align*}
			can be written as 
			\begin{align}
				\label{eq:example_f_b_1_good}
				f(x) = \hf x \sqrt{x^{2} - 2 a} - a \log\left(\sqrt{x^{2} - 2a} + x\right) + b,
			\end{align} 
			which satisfies \cref{eq:lip_special} with $\beta = L=1$. It can be easily verified that 
			\begin{align*}
				f''(x) = \frac{x}{\sqrt{x^{2} - 2 a}}, \quad \text{ and } \quad f'''(x) = -\frac{2 a}{{\left(x^2 - 2 a\right)}^{\frac{3}{2}}},
			\end{align*}
			which are unbounded, and so,  $f$ does not satisfy \cref{eq:lip_usual_grad} or \cref{eq:lip_usual_hessian}. It is clear that, on $ (x_{0}, \infty) $, \cref{eq:example_f_b_1_good} is convex, and hence it is invex.
			
			Similarly, for $ b \in \mathbb{R}, a \in \mathbb{R}^{++} $, one can consider 
			\begin{align*}
				f''(x) f'(x) = -x,  \quad x \in (-\sqrt{2 a}, \sqrt{2 a}), \quad f(0) = b, \quad f'(0) = -\sqrt{2a}.
			\end{align*}
			The solution to this ODE is of the form
			\begin{align}
				\label{eq:example_f_b_1_bad}
				f(x) = b - \frac{x}{2} \sqrt{2 a - x^{2}} - a \arctan\left(\frac{x}{\sqrt{2 a - x^{2}}}\right),
			\end{align} 
			which satisfies \cref{eq:lip_special} with $\beta = L=1$.
			However, since 
			\begin{align*}
				f''(x) = \frac{x}{\sqrt{2 a -x^{2}}}, \quad \text{ and } \quad f'''(x) = \frac{2 a}{{\left(2a - x^2\right)}^{\frac{3}{2}}},
			\end{align*}
			are both unbounded, such a function does not satisfy either of \cref{eq:lip_usual_grad} or \cref{eq:lip_usual_hessian}.
			
			\item For $\beta > 1$, the condition \cref{eq:lip_special} implies that $\left[\nabla^{2}f(\xx)\right]^{2} + \dotprod{\nabla^{3}f(\xx),\bgg(\xx)} = 0$.
			In one variable, the solution to the following ODE 
			\begin{align*}
				\left( f''(x) \right)^{2} + f'''(x) f'(x) = 0, \quad x > c/2, \quad f(c/2) = a, \quad f'(c/2) = 0, \quad f''(c) = 3b/\sqrt{c}.
			\end{align*}
			can be written as 
			\begin{align}
				\label{eq:example_f_b_g_1}
				f(x) = a + b (2 x - c)^{3/2}.
			\end{align}
		\end{subequations}
		This implies that there are non-trivial function which can satisfy \cref{eq:lip_special} for $\beta > 1$. In addition, since
		\begin{align*}
			f''(x) = \frac{3b}{(2x - c)^{1/2}},\quad f'''(x) = \frac{-3b}{(2x - c)^{3/2}},
		\end{align*}
		are both unbounded, it implies that such $f$ does not satisfy either assumptions in \cref{eq:lip_usual}. It is clear that, on $ (c/2, \infty) $, \cref{eq:example_f_b_g_1} is convex, and hence it is invex.
	\end{enumerate}
\end{example}


\cref{assmpt:lipschitz_special} allows for Hessian to grow unboundedly, however it implies a certain restriction on the growth rate of the action of Hessian on the gradient, i.e., $\HH(\xx) \bgg(\xx)$.
\begin{lemma}[Growth Rate of $\HH(\xx) \bgg(\xx)$]
	\label{lemma:growth_rate}
	Under \cref{assmpt:lipschitz_special}, for any $ \xx_{0} \in \real^{d}$, we have
	\begin{align*}
		L(\xx_{0}) \geq \left(\frac{2 \beta}{\beta+1}\right)^{\beta} \frac{\vnorm{\HH(\xx) \bgg(\xx)}^{\beta+1}}{\vnorm{\bgg(\xx)}^{2\beta}}, \quad \forall \xx \in \mathcal{X}_{0},
	\end{align*}
	where $L(\xx_{0}), \beta$ and $ \mathcal{X}_{0} $ are as defined in \cref{assmpt:lipschitz_special}.
\end{lemma}
\begin{proof}
	If $ L(\xx_{0}) = 0 $, then the result follows immediately. Suppose $ L(\xx_{0}) > 0 $. Under \cref{assmpt:lipschitz_special}, \cref{lemma:aux} with $ h(\xx) = \vnorm{\bgg(\xx)}^{2}/2 $, gives
	\begin{align*}
		\vnorm{\bgg(\xx + \pp)}^{2} &\leq \vnorm{\bgg(\xx)}^{2} + 2  \dotprod{\HH(\xx) \bgg(\xx), \pp} + \frac{2 L(\xx_{0})}{(\beta+1)} \vnorm{ \pp }^{\beta+1}, \quad \forall \xx \in \mathcal{X}_{0}, \quad \forall \pp \in \real^{d}.
	\end{align*}
	Let $m(\pp) \triangleq \vnorm{\bgg(\xx)}^{2} + 2  \dotprod{\HH(\xx) \bgg(\xx), \pp} + {2 L(\xx_{0})}/{(\beta+1)} \vnorm{ \pp }^{\beta+1}$.
	Consider $ \pps \neq \bm{0} $ such that 
	\begin{align*}
		\nabla m(\pps) = 2 \HH(\xx) \bgg(\xx) + 2 L(\xx_{0}) \vnorm{ \pps }^{\beta-1} \pps = 0.
	\end{align*}
	It implies that $\vnorm{ \pps }^{\beta-1} \pps = -\HH(\xx) \bgg(\xx)/ L(\xx_{0})$. As a result, for such $ \pps $, we must have $\vnorm{ \pps }^{\beta} = \vnorm{\HH(\xx) \bgg(\xx)}/ L(\xx_{0})$, and $\dotprod{\HH(\xx) \bgg(\xx), \pps} = -\vnorm{\HH(\xx) \bgg(\xx)} \vnorm{\pps}$, where the equality follows since $ \pps $ is a scalar multiple of $ \HH(\xx) \bgg(\xx) $. Since $ m(\pp) $ is convex, it follows that 
	\begin{align*}
		\min_{\pp} m(\pp) &= \vnorm{\bgg(\xx)}^{2} - 2 \vnorm{\HH(\xx) \bgg(\xx)} \vnorm{\pps} + \frac{2 L(\xx_{0})}{(\beta+1)} \vnorm{ \pps }^{\beta+1} \\
		&= \vnorm{\bgg(\xx)}^{2} - \frac{2 \vnorm{\HH(\xx) \bgg(\xx)}^{\frac{\beta+1}{\beta}}}{L(\xx_{0})^{\frac{1}{\beta}}} + \frac{2 \vnorm{\HH(\xx) \bgg(\xx)}^{\frac{\beta+1}{\beta}} }{(\beta+1) L(\xx_{0})^{\frac{1}{\beta}}} \\
		&= \vnorm{\bgg(\xx)}^{2} - \left(\frac{\beta}{\beta+1}\right)\left(\frac{2}{L(\xx_{0})^{\frac{1}{\beta}}}\right) \vnorm{\HH(\xx) \bgg(\xx)}^{\frac{\beta+1}{\beta}} .
	\end{align*}
	The result follows since we have $ m(\pp) \geq \vnorm{\bgg(\xx + \pp)}^{2} \geq 0, \forall \pp $.
\end{proof}

We also require the following regularity on the pseudo-inverse of the Hessian matrix.
\begin{assumption}[Pseudo-inverse Regularity]
	\label{assmpt:pseudo_regularity}
	For any $ \xx_{0} \in \real^{d} $, there is a $\gamma(\xxo) > 0$, s.t.\
	\vspace{-1mm}	
	\begin{align}
		\label{eq:pseudo_regularity}
		\|[\HH(\xx)]^{\dagger}\| \leq 1/\gamma(\xxo), \quad \forall \xx \in \mathcal{X}_{0},
	\end{align}
	where $ \mathcal{X}_{0} $ is as in \cref{assmpt:lipschitz_special}.
\end{assumption}
It turns out that \cref{eq:pseudo_regularity} is equivalent to the regularity of the Hessian matrix on its \emph{range space}.
\vspace{-2mm}
\begin{lemma}
	\label{lemma:pseudo_regularity}
	\cref{eq:pseudo_regularity} is equivalent to 
	\vspace{-2mm}
	\begin{align}
		\label{eq:pseudo_regularity_lower}
		\vnorm{\HH(\xx) \pp \geq \gamma(\xxo) \vnorm{\pp}, \quad \forall \xx \in \mathcal{X}_{0}, \quad \text{and} \quad \forall \pp \in \text{Range}\left(\HH(\xx)\right).}
	\end{align}
\end{lemma}
The proof of \cref{lemma:pseudo_regularity} follows simply by considering the eigenvalue decomposition of $ \HH $.
\cref{assmpt:pseudo_regularity} is trivially satisfied for all strongly convex functions. However, it is also satisfied for any function whose (possibly rank-deficient) Hessian is uniformly positive definite in the sup-space spanned by its range. 

\begin{example}[Functions satisfying \cref{assmpt:pseudo_regularity}]
	A simple example of a non-strongly convex function satisfying \cref{assmpt:pseudo_regularity} is the undetermined least squares $ f(\xx) = \vnorm{\AA\xx - \bb}^{2}/2 $ with full row rank matrix $ \AA \in \real^{n \times d}, ~ n \leq d $. This problem is clearly only weakly convex since the Hessian matrix, $ \AA^{\intercal} \AA \in \real^{d \times d}$, is rank-deficient. However, it is easy to see that \cref{eq:pseudo_regularity} holds with $ \gamma(\xxo) = \gamma = \sigma_{n}^{2}(\AA) $, where $\sigma_{n}$ is the smallest non-zero singular value of $ \AA $.
	
\end{example}

\vspace{-1mm}
Finally, we make the following structural assumption about $ f $ with regards to the gradient and its projection onto the range space of Hessian. 

\begin{assumption}[Gradient-Hessian Null-Space Property]
	\label{assmpt:null_space}
	For any $\xx \in \mathbb{R}^{d}$, let $\UU_{\xx}$ and $\UU^{\perp}_{\xx}$ denote arbitrary orthogonal bases for $\text{Range}(\HH(\xx))$ and its orthogonal complement, respectively. A function is said to satisfy the Gradient-Hessian Null-Space property, if for any $ \xxo \in \real^{d} $, there is a $0 < \nu(\xxo) \leq 1$, such that 
	\vspace{-2mm}
	\begin{align}
		\label{eq:null_space}
		\nu(\xxo) \vnorm{\left(\UU^{\perp}_{\xx}\right)^{\intercal} \bgg(\xx)}^{2} \leq \big( 1 - \nu(\xxo) \big) \vnorm{\UU_{\xx}^{\intercal} \bgg(\xx)}^{2}, \quad \forall \xx \in \mathcal{X}_{0},
	\end{align}
	where $ \mathcal{X}_{0} $ is as in \cref{assmpt:lipschitz_special}.
\end{assumption}

\cref{lemma:null_space_grad} gives some simple, yet useful, consequences of \cref{assmpt:null_space}.
\begin{lemma}
	\label{lemma:null_space_grad}
	Under \cref{assmpt:null_space}, we have 
	\begin{subequations}
		\label{eq:null_space_grad}	
		\begin{align}
			\nu(\xxo) \vnorm{\bgg(\xx)}^{2} &\leq  \vnorm{\UU_{\xx}^{\intercal}\bgg(\xx)}^{2}, \quad \quad \quad \forall \xx \in \mathcal{X}_{0}, \label{eq:null_space_grad_1} \\
			\big(1-\nu(\xxo)\big) \vnorm{\bgg(\xx)}^{2} &\geq  \vnorm{\left(\UU^{\perp}_{\xx}\right)^{\intercal}\bgg(\xx)}^{2}, \quad \; \forall \xx \in \mathcal{X}_{0}. \label{eq:null_space_grad_2}
		\end{align}
	\end{subequations}
\end{lemma}

\begin{proof}
	For a given $ \xxo \in \real^{d} $, take any $ \xx \in \mathcal{X}_{0} $. For \cref{eq:null_space_grad_1}, we have
	\begin{align*}
		\vnorm{\bgg(\xx)}^{2} &= \vnorm{\left(\UU_{\xx} \UU_{\xx}^{\intercal} + \UU^{\perp}_{\xx} \left(\UU^{\perp}_{\xx}\right)^{\intercal}\right) \bgg(\xx)}^{2} = \vnorm{\UU_{\xx}^{\intercal}\bgg(\xx)}^{2} + \vnorm{\left(\UU^{\perp}_{\xx}\right)^{\intercal} \bgg(\xx)}^{2} \\
		& \leq \frac{1 - \nu(\xxo)}{\nu(\xxo)} \vnorm{\UU_{\xx}^{\intercal}\bgg(\xx)}^{2} +  \vnorm{\UU_{\xx}^{\intercal}\bgg(\xx)}^{2} = \frac{1}{\nu(\xxo)} \vnorm{\UU_{\xx}^{\intercal}\bgg(\xx)}^{2}.
	\end{align*}
	Also, \cref{eq:null_space_grad_2} is obtained similarly.  
\end{proof}

\cref{assmpt:null_space} ensures that, as long as the gradient is non-zero, its angle with the sub-space spanned by the Hessian matrix is uniformly bounded away from zero. In other words, the gradient will never become arbitrarily orthogonal to the range space of Hessian.

\begin{example}[Over-parameterized ERM Using Linear Predictor Models]
	\label{example:erm}
	Consider an empirical risk minimization (ERM) problem involving linear predictor models \cite{shalev2014understanding},
	\begin{align}
		\label{eq:erm}
		f(\xx) = \sum_{i=1}^{n} f_{i}(\aa_{i}^{\intercal}\xx),
	\end{align}
	where $ \aa_{i} \in \real^{d}, i =1,\ldots n$  are given data points and each $ f_{i}:\real \rightarrow \real $ is some nonlinear misfit or loss corresponding to the $ i\th $ data point. Consider the over-parameterized settings ($ n \leq d $), which has recently garnered significant attention within the machine learning community, e.g., \cite{taheri2021fundamental,su2019learning,muthukumar2021classification,liu2022loss}. Assume that the data points are linearly independent, i.e., $ \text{Range}\left( \left\{ \aa_{i} \right\}_{i=1}^{n} \right) = \real^{n} $. Further, suppose each loss function $ f_{i} $ is such that if $ f_{i}^{''}(t) = 0$ then we must also have that $f_{i}^{'}(t) = 0 $. Many loss functions satisfy this assumption, e.g., the convex function $ t^{2p} $ for any $ p \geq 1 $, so that $ f_{i}(\aa_{i}^{\intercal}\xx) = \dotprod{\aa_{i},\xx}^{2p} $. Another example is when each $ f_{i} $ is such that $f_{i}^{''}(t) > 0, \forall t $, e.g., logistic function $ \log(1+e^{-t}) $ or any strongly convex. Note that even if each $ f_{i} $ is strongly convex, since we have $ n \leq d $, the overall objective $ f $ in \cref{eq:erm} is still only weakly convex. 
	
	Define
	\begin{align*}
		\AA = \begin{pmatrix}
			\aa_{1}^{\intercal} \\
			\vdots \\
			\aa_{n}^{\intercal}
		\end{pmatrix}  \in \mathbb{R}^{n \times d},\quad 
		\DD = \begin{pmatrix}
			f''(\aa_{1}^{\intercal} \xx) & & & \\
			& f''(\aa_{2}^{\intercal} \xx) & & \\
			& & \ddots & \\
			& & & f''(\aa_{n}^{\intercal} \xx)
		\end{pmatrix} \in \mathbb{R}^{n \times n}.
	\end{align*}
	It is easy to see that
	\begin{align*}
		\bgg(\xx) = \AA^{\intercal} \begin{pmatrix} f_{i}^{'}(\aa_{1}^{\intercal} \xx) \\ \vdots \\ f_{i}^{'}(\aa_{n}^{\intercal} \xx) \end{pmatrix}, \quad 	\HH(\xx) = \AA^{\intercal} \DD \AA.
	\end{align*}
	For the gradient to be in the range of the Hessian, we must have for some $ \vv \in \real^{d} $
	\begin{align}
		\label{eq:erm_nullspace}
		\AA^{\intercal} \DD \AA \vv = \AA^{\intercal} [f'(\aa_{1}^{\intercal} \xx),\ldots,f'(\aa_{n}^{\intercal} \xx)]^{\intercal}.
	\end{align}
	Under the assumption on $ f_{i} $, it can be seen that $\vv = \AA^{\dagger} \DD^{\dagger} [f'(\aa_{1}^{\intercal} \xx),\ldots,f'(\aa_{n}^{\intercal} \xx)]^{\intercal}$, 	satisfies \cref{eq:erm_nullspace}, 	where $\AA^{\dagger} = \AA^{\intercal} \left(\AA \AA^{\intercal}\right)^{-1} \in \mathbb{R}^{d \times n}$. Hence, \cref{assmpt:null_space} holds with $ \nu(\xxo) = \nu = 1 $. 
\end{example}

\subsection{Convergence Analysis}
\label{sec:convergence}
In this section, we discuss the convergence properties of Newton-MR. For this, in \cref{sec:exact}, we first consider the slightly simpler, yet less practical, case where the sub-problems \cref{eq:least_norm_solution} are solved exactly (\cref{alg:NewtonMR_Ex}). Clearly, this is too stringent, and as a result, in \cref{sec:inexact}, it is subsequently relaxed to allows inexact solutions (\cref{alg:NewtonMR}). Convergence under a generalized variant of Polyak-\L{}ojasiewicz inequality will be treated in \cref{sec:gpl}.

The following simple Lemma relating to \cref{eq:lip_special} is frequently used in our theoretical analysis. The proof can be found in most textbooks and is only given here for completeness.
\begin{lemma}
	\label{lemma:aux}
	Consider any $ \xx,\zz \in \mathbb{R}^{d} $, $0 < \beta < \infty$, $ 0 \leq L < \infty $ and $ h: \mathbb{R}^{d} \rightarrow \mathbb{R} $. If $$\vnorm{\nabla h(\yy) - \nabla h(\xx)} \leq L \vnorm{\yy - \xx}^{\beta}, \; \forall \yy \in \left[\xx, \zz\right],$$ then $$h(\yy) \leq h(\xx) + \dotprod{\nabla h(\xx), \yy - \xx} + L \vnorm{\yy - \xx}^{\beta+1}/(\beta+1), \; \forall \yy \in \left[\xx, \zz\right].$$
\end{lemma}

\begin{proof}
	For any $ \yy \in \left[\xx, \zz\right] $, using the mean value theorem, we have
	\begin{align*}
		h(\yy) - h(\xx) - \dotprod{\nabla h(\xx), \yy - \xx} &= \int_{0}^{1}  \dotprod{\nabla h(\xx + \tau (\yy - \xx)), \yy - \xx} d \tau - \dotprod{\nabla h(\xx), \yy - \xx} \\
		&\leq  \vnorm{\yy - \xx}  \int_{0}^{1}  \vnorm{\nabla h(\xx + \tau (\yy - \xx)) - \nabla h(\xx)} d \tau \\
		& \leq L \vnorm{\yy - \xx}^{\beta+1}  \int_{0}^{1}  \tau^{\beta} d \tau \leq \frac{L}{\beta+1} \vnorm{\yy - \xx}^{\beta+1}.
	\end{align*}
\end{proof}

\subsubsection{Exact Update}
\label{sec:exact}
The underlying sub-problem of Newton-MR, at $k\th$ iteration, involves OLS problem of the form \cref{eq:ols}. Under the invexity assumption, the Hessian matrix can be indefinite and rank deficient, and as a result, the sub-problem \eqref{eq:ols} may contain infinitely many solutions.  
Indeed, the general solution of \eqref{eq:ols} is written as 
\begin{align*}
	\pp = -\left[\HHk\right]^{\dagger} \bggk + \left(\eye - \HHk \left[\HHk\right]^{\dagger}\right) \qq, \quad \forall \qq \in\mathbb{R}^{d},
\end{align*}
which implies that when $ \HHk $ has a non-trivial null-space, the sub-problem \eqref{eq:ols} has infinitely many solutions. 
Among these, the one with the minimum norm is defined uniquely as 
\begin{align}
	\label{eq:least_norm_solution}
	\min_{\pp \in \real^{d}} ~~ \|\pp\| \quad \text{s.t.} \quad \pp \in \Argmin_{\widehat{\pp} \in \mathbb{R}^{d}} \vnorm{\bggk + \HHk \widehat{\pp}},
\end{align}
which yields $ \ppk = -\left[\HHk\right]^{\dagger} \bggk $. 
In our analysis below, among all possible solutions, we will choose the least norm solution, to get iterations of the form \cref{eq:newton_mr_iterations}. 
The resulting Newton-MR variant with exact update \cref{eq:least_norm_solution} is depicted in \cref{alg:NewtonMR_Ex}.

\begin{algorithm}
	\caption{Exact Newton-MR}
	\begin{algorithmic}
		\vspace{1mm}
		\STATE \textbf{Input:} 
		\begin{itemize}[label=-]
			\item Initial iterate $\xx_{0} \in \real^{d}$, Line-search parameter $0 < \rho < 1$, Termination tolerance $0 < \epsilon$
		\end{itemize}
		\vspace{1mm}
		\FOR {$k = 0,1,2, \cdots$ until $ \vnorm{\bggk} \leq \epsilon $} 
		\vspace{1mm}
		\STATE Find $ \ppk $ using \cref{eq:least_norm_solution}, i.e., set $\ppk = -\left[\HHk\right]^{\dagger} \bggk$
		\vspace{1mm}
		\STATE \label{alg:step:alphak} Find step-size, $\alphak$, such that \eqref{eq:armijo_gen} holds with $ \rho $
		\vspace{1mm}
		\STATE Set $\xxkk  =  \xxk + \alphak \ppk$
		\vspace{1mm}
		\ENDFOR
		\vspace{1mm}
		\STATE \textbf{Output:} $\xx$ for which $ \vnorm{\bgg(\xx)} \leq \epsilon$
	\end{algorithmic}
	\label{alg:NewtonMR_Ex}
\end{algorithm}

\cref{thm:exact} gives the convergence guarantees of \cref{alg:NewtonMR_Ex}.

\begin{theorem}[Convergence of \cref{alg:NewtonMR_Ex}]
	\label{thm:exact}
	Consider \cref{assmpt:diff,assmpt:lipschitz_special,assmpt:pseudo_regularity,assmpt:null_space}. For the iterates of \cref{alg:NewtonMR_Ex}, we have 
	\vspace{-1mm}
	\begin{align*}
		\vnorm{\bggkk}^2 \le \left( 1 - 2 \rho \tau(\xxo) \vnorm{\bggk}^{(1-\beta)/\beta} \right)\vnorm{\bggk}^2,
	\end{align*}
	where $\tau(\xxo) \triangleq \left( {(1-\rho) (1+\beta) \left(\nu(\xxo) \gamma(\xxo)\right)^{1+\beta} }/{L(\xx_{0})} \right)^{1/\beta}$, $ \rho $ is the line-search parameter of \cref{alg:NewtonMR_Ex}, $ \xx_{0} $ is the initial iterate, $ (\beta , L(\xx_{0})) $ are as in \cref{assmpt:lipschitz_special}, $ \gamma(\xxo) $ is as in \cref{assmpt:pseudo_regularity}, and $ \nu(\xxo) $ is defined in \cref{assmpt:null_space}.
\end{theorem}

\begin{proof}
	We first note that, by \cref{eq:grad_descent}, all iterates of \cref{alg:NewtonMR_Ex} remain in $ \sX_{0} $. Let $\UUk \in \mathbb{R}^{d \times r}$ be any orthogonal basis for the range of $\HHk$ and $r = \text{rank}\left(\HHk\right) \leq d$. It follows that $\dotprod{\HHk \bggk, \ppk} = -\dotprod{\bggk, \HHk \left[\HHk\right]^{\dagger} \bggk} =- \vnorm{\UU_{k}^{\intercal}\bggk}^{2}$. 
	Using \eqref{eq:armijo_gen}, we get the reduction in the gradient norm as
	\begin{align}
		\label{eq:grad_decrease}
		\vnorm{\bggkk}^{2} \leq \vnorm{\bggk}^{2} - 2 \rho \alphak \vnorm{\UU_{k}^{\intercal} \bggk}^{2}.
	\end{align}
	All that is left is to obtain a non-zero iteration independent lower-bound on $\alphak$, i.e., $ \alphak \geq \alpha > 0$, for which \cref{eq:grad_decrease} holds. For this, using \cref{assmpt:lipschitz_special} and \cref{lemma:aux}, with $ \xx = \xxk $, $ \zz = \xxk + \ppk $, $ \yy = \xxk + \alpha \ppk $, and $ h(\xx) = \vnorm{\bgg}^{2}/2 $, we get
	\begin{align}
		\label{eq:grad_upper_bound}
		\vnorm{\bggkk}^{2} &\leq \vnorm{\bggk}^{2} + 2 \alpha \dotprod{\HHk \bggk, \ppk} + \frac{2 \alpha^{\beta+1} L(\xx_{0})}{(\beta+1)} \vnorm{ \ppk }^{\beta+1}.
	\end{align}
	\cref{assmpt:pseudo_regularity} implies $\vnorm{ \ppk } = \vnorm{ \left[\HHk\right]^{\dagger} \bggk } \leq \vnorm{ \bggk }/[\gamma(\xxo)]^{\beta+1}$. Hence, it follows that
	\begin{align*}
		\vnorm{\bggkk}^{2} &\leq \vnorm{\bggk}^{2} - 2\alpha \vnorm{\UU_{k}^{\intercal} \bggk}^{2} + \frac{ 2 \alpha^{\beta+1} L(\xx_{0})}{[\gamma(\xxo)]^{\beta+1} (\beta+1)} \vnorm{\bggk}^{\beta+1}.
	\end{align*}
	Using the above as well as \cref{eq:grad_decrease}, we require for $ \alpha $ to satisfy 
	\begin{align*}
		\vnorm{\bggk}^{2} - 2\alpha \vnorm{\UU_{k}^{\intercal} \bggk}^{2} + \frac{ 2 \alpha^{\beta+1} L(\xx_{0})}{[\gamma(\xxo)]^{\beta+1} (\beta+1)} \vnorm{\bggk}^{\beta+1} \leq  \vnorm{\bggk}^2 - 2\rho \alpha \vnorm{\UU_{k}^{\intercal} \bggk}^{2}.
	\end{align*}
	But from \cref{eq:null_space_grad_1} and \cref{assmpt:null_space}, this is implied if $ \alpha $ satisfies $\alpha^{\beta+1} L(\xx_{0}) \vnorm{\bggk}^{\beta+1} \leq  [\gamma(\xxo)]^{\beta+1} (\beta+1) (1-\rho)\nu(\xxo) \alpha \vnorm{\bggk}^2$, which is given if $\alpha \leq \left( {(1-\rho) (1+\beta) \nu(\xxo) [\gamma(\xxo)]^{1+\beta} }/{L(\xx_{0})} \right)^{1/\beta} \vnorm{\bggk}^{(1-\beta)/\beta}$. 
	This implies that any step-size returned from the line-search will be at least as large as the right-hand side. Now, using \cref{assmpt:null_space} and \cref{lemma:null_space_grad} again, we get $\vnorm{\bggkk}^2 \leq \vnorm{\bggk}^2  - 2\rho\alpha \vnorm{\UU_{k}^{\intercal}\bggk}^2 \leq (1 - 2\rho\nu(\xxo) \alpha) \vnorm{\bggk}^2.$ 
\end{proof}

\begin{remark}
	\label{rem:meaningful_rate}
	The convergence rate of \cref{thm:exact}, for $ \beta \neq 1 $, seems rather complicated. However, one can ensure that it is indeed meaningful, i.e., the rate is positive and strictly less than one. From \cref{lemma:growth_rate,assmpt:pseudo_regularity,assmpt:null_space}, it follows that $ \forall \xx \in \mathcal{X}_{0} $, we have
	\begin{align*}
		L(\xx_{0}) &\geq \left(\frac{2 \beta}{\beta+1}\right)^{\beta} \frac{\vnorm{\HH(\xx) \bgg(\xx)}^{\beta+1}}{\vnorm{\bgg(\xx)}^{2\beta}} = \left(\frac{2 \beta}{\beta+1}\right)^{\beta} \frac{\vnorm{\HH(\xx) \left(\UU_{\xx} \UU_{\xx}^{\intercal} \bgg(\xx)\right)}^{\beta+1}}{\vnorm{\bgg(\xx)}^{2\beta}} \\
		&\geq \left(\frac{2 \beta}{\beta+1}\right)^{\beta} (\gamma(\xxo) \sqrt{\nu(\xxo)})^{\beta+1} \vnorm{ \bgg(\xx)}^{1-\beta},
	\end{align*}
	where $\UU_{\xx} \in \mathbb{R}^{d \times r}$ is any orthogonal basis for the range of $\HH(\xx)$. 
	It can also be shown that 
	\begin{align*}
		\rho \left((1+\beta)(1-\rho)\right)^{1/\beta} \leq \frac{\beta}{1+\beta}, ~\forall \rho \in (0,1), ~\forall \beta \in (0,\infty),
	\end{align*}
	with equality holding at $ \rho = \beta/(1+\beta) $.
	Hence, noting that $ \nu(\xxo) \leq 1 $, we obtain
	\begin{align*}
		L(\xx_{0}) \geq 2^{\beta} \rho^{\beta} (1-\rho)(1+\beta) \left(\nu(\xxo) \gamma(\xxo)\right)^{\beta+1} \vnorm{ \bgg(\xx)}^{1-\beta}, \quad \forall \xx \in \mathcal{X}_{0},
	\end{align*}
	which ensures 
	\begin{align}
		\label{eq:meaningful_rate}
		0 \leq 1 - 2 \rho \tau(\xxo) \vnorm{\bggk}^{(1-\beta)/\beta} < 1, \quad k = 0,1,\ldots.
	\end{align}
\end{remark}


\begin{remark}[Parallels between Newton-MR/Newton-CG and MINRES/CG]
	\label{rem:traditional}
	In the traditional settings of strongly convex and smooth functions, the global convergence rate in \cref{thm:exact} with $ \beta = 1 $ is identical to that of Newton-CG; see \cite{ssn2018} for example. However, the former indicates the rate of reduction in the norm of the gradient of the function, while the latter corresponds to the reduction in the function value itself. This relationship is reminiscent of the convergence guarantees of MINRES and CG for linear systems involving symmetric positive-definite matrices. Indeed, while having identical convergence rates, the former is stated in terms of the reduction in the norm of the residual of the iterates, while the latter indicates the reduction in the error of the iterates themselves. Also, recall that MINRES, unlike CG, can be applied beyond positive-definite matrices. In the same spirit, Newton-MR, unlike Newton-CG, is applicable beyond the traditional convex settings to invex functions.
	
\end{remark}

Similarly as in the case of many Newton-type methods, we can obtain local convergence guarantees for Newton-MR with unit step-size, i.e., $ \alphak = 1 $. We further show that such result greatly generalizes the classical analysis of Newton-CG. 
We will show that local (super-)linear rate of convergence is possible under either \cref{assmpt:lipschitz_special} with $ \beta > 1 $, or \cref{assmpt:lip_special_03}. 

\begin{assumption}
	\label{assmpt:lip_special_03}
	For some $ 0 < \beta < \infty $ and $ 0 \leq L_{\HH} < \infty $, we have $ \forall \xx \in \real^{d},\; \forall \pp \in \text{Range}(\HH(\xx))$,
	\begin{align}
		\label{eq:lip_special_03}
		\dotprod{\bgg(\xx + \pp),\big( \HH(\xx + \pp) - \HH(\xx) \big) \pp } &\leq L_{\HH} \vnorm{\bgg(\xx + \pp)}\vnorm{\pp}^{1+\beta}. 
	\end{align}
\end{assumption}
Although, we do not know of a particular way to, a priori, verify \cref{assmpt:lip_special_03}, it is easy to see that \cref{eq:lip_special_03} with $ \beta = 1 $ is implied by \cref{eq:lip_usual_hessian}, and hence weaker. In fact, unlike the usual local convergence analysis of Newton-type methods, analyzing iterations in terms of gradient norm allows us to weaken \cref{eq:lip_usual_hessian} and instead consider \cref{eq:lip_special_03}. 
%

\begin{theorem}[Error Recursion of \cref{alg:NewtonMR_Ex} with $ \alphak = 1 $]
	\label{thm:exact_local}
	Consider Assumptions \ref{assmpt:diff}, \ref{assmpt:pseudo_regularity}, and \ref{assmpt:null_space}. Suppose $ \xxk \in \sX_{0} $ and consider one iteration of \cref{alg:NewtonMR_Ex} with $ \alphak = 1 $.
	\begin{enumerate}[label = (\roman*)]
		\item \label{cond:thm:exact_local_01} If \cref{assmpt:lipschitz_special} holds, then
		\begin{align*}
			\vnorm{\bggkk}^{2} \leq \frac{2 L(\xx_{0})}{(1+\beta)[\gamma(\xxo)]^{1+\beta}} \vnorm{\bggk}^{1+\beta} + (1-2\nu(\xxo)) \vnorm{\bggk}^{2}.
		\end{align*}
		\item \label{cond:thm:exact_local_02} If \cref{assmpt:lip_special_03} holds and $ \HH(\xx) $ is continuous, then
		\begin{align*}
			\vnorm{\bggkk} \leq \frac{L_{\HH}}{(1+\beta) [\gamma(\xxo)]^{1+\beta}} \vnorm{\bggk}^{1+\beta}  + \sqrt{1-\nu(\xxo)} \vnorm{\bggk}.
		\end{align*}
	\end{enumerate}
	Here, $ L(\xx_{0}), \gamma(\xxo), \nu(\xxo) $ and $ L_{\HH} $ are defined, respectively, in Assumptions \ref{assmpt:lipschitz_special}, \ref{assmpt:pseudo_regularity}, \ref{assmpt:null_space}, and \ref{assmpt:lip_special_03}. Also, $\beta$ refers to the respective constants of Assumptions  \ref{assmpt:lipschitz_special} and \ref{assmpt:lip_special_03}.
\end{theorem}
\begin{proof}
	Recall that \cref{eq:least_norm_solution} and $ \alphak = 1 $ implies that $ \xxkk = \xxk + \ppk $, where we have $ \ppk = - \left[\HH \left( \xxk \right)\right]^{\dagger} \bgg \left( \xxk \right) $. Throughout the proof, let $\UU_{\xx}$ and $\UU^{\perp}_{\xx}$ denote any orthogonal bases for $\text{Range}(\HH(\xx))$ and its orthogonal complement, respectively.
	\noindent 
	(i) \; From \cref{eq:grad_upper_bound} with $ \alphak = 1 $ and using Assumptions \ref{assmpt:pseudo_regularity} and \ref{assmpt:null_space}, we get
		\begin{align*}
			\vnorm{\bggkk}^{2} &\leq \vnorm{\bggk}^{2} - 2\vnorm{\UU_{k}^{\intercal} \bggk}^{2} + \frac{ 2  L(\xx_{0})}{(\beta+1)[\gamma(\xxo)]^{\beta+1}} \vnorm{\bggk}^{\beta+1} \\
			&\leq (1-2\nu(\xxo)) \vnorm{\bggk}^{2} + \frac{2 L(\xx_{0})}{(1+\beta)[\gamma(\xxo)]^{1+\beta}} \vnorm{\bggk}^{1+\beta}.
		\end{align*}
		
	\noindent 
	(ii) \; Since by assumption $ \bgg(\xx) $ is continuously differentiable, using mean-value theorem for vector-valued functions (\cite[Theorem 7.9-1(d)]{ciarlet2013linear}) for $ \bggkk = \bgg \left(\xxk + \ppk \right) $, we have
	\begin{align}
		\label{eq:mvt}
		\vnorm{\bggkk}^{2} &= \dotprod{\bggkk,\bggkk} = \dotprod{\bggkk,\bggk + \int_{0}^{1} \left[\HH\left(\xxk + t \ppk \right) \ppk \right] \df t }.
	\end{align}
	Note that
	\begin{align*}
		\bgg \left( \xxk \right) &= \HH \left( \xxk \right) \left[\HH \left( \xxk \right)\right]^{\dagger} \bgg \left( \xxk \right) + \UU_{\xx}^{\perp} \left[\UU_{\xx}^{\perp}\right]^{\intercal} \bgg \left( \xxk \right)  = -\HH \left( \xxk \right) \ppk + \UU_{\xx}^{\perp} \left[\UU_{\xx}^{\perp}\right]^{\intercal} \bgg \left( \xxk \right). 
	\end{align*}
	Hence, it follows that 
	\begin{align*}
		\vnorm{\bggkk}^{2} &= \dotprod{\bggkk, -\HH \left( \xxk \right) \ppk + \int_{0}^{1} \left[\HH\left(\xxk + t \ppk \right) \ppk \right] \df t } + \dotprod{\bggkk, \UU_{\xx}^{\perp} \left[\UU_{\xx}^{\perp}\right]^{\intercal} \bgg \left( \xxk \right)} \\
		&\leq \dotprod{\bggkk, \int_{0}^{1} \left[ \left( \HH\left(\xxk + t \ppk \right) -\HH \left( \xxk \right) \right) \ppk \right] \df t } + \vnorm{\bggkk}\vnorm{\left[\UU_{\xx}^{\perp}\right]^{\intercal} \bgg \left( \xxk \right)} \\
		&\leq \int_{0}^{1} t^{-1} \left[ \dotprod{\bggkk, \left( \HH\left(\xxk + t \ppk \right) -\HH \left( \xxk \right) \right) t \ppk} \right] \df t + \sqrt{1-\nu(\xxo)} \vnorm{\bggkk} \vnorm{\bggk},
	\end{align*}
	where the inequality follows by \cref{lemma:null_space_grad}. Using Assumptions \ref{assmpt:pseudo_regularity} and \ref{assmpt:lip_special_03}, we get
	\begin{align*}
		\vnorm{\bggkk}^{2} &\leq L_{\HH} \vnorm{\bggkk} \vnorm{\ppk}^{1+\beta} \int_{0}^{1} t^{\beta} \df t   + \sqrt{1-\nu(\xxo)} \vnorm{\bggkk} \vnorm{\bggk},
	\end{align*}
	and hence,
	\begin{align*}
		\vnorm{\bggkk} &\leq \frac{L_{\HH}}{(1+\beta)} \vnorm{\ppk}^{1+\beta} + \sqrt{1-\nu(\xxo)} \vnorm{\bggk} \leq \frac{L_{\HH}}{(1+\beta)} \vnorm{\left[\HHk\right]^{\dagger} \bggk}^{1+\beta}   + \sqrt{1-\nu(\xxo)} \vnorm{\bggk} \\
		& \leq \frac{L_{\HH}}{(1+\beta) [\gamma(\xxo)]^{1+\beta}} \vnorm{\bggk}^{1+\beta}  + \sqrt{1-\nu(\xxo)} \vnorm{\bggk}. 
	\end{align*}
\end{proof}

\begin{remark}
	\label{rem:local_super_linear}
	\cref{thm:exact_local}-\ref{cond:thm:exact_local_02} implies convergence in gradient norm with a local rate that is linear when $ \nu(\xxo) < 1 $, super-linear when $ \nu(\xxo) = 1 $, and quadratic when $ \nu(\xxo) = \beta = 1 $. For example, when $ \nu(\xxo) < 1 $, for any given $ \sqrt{1-\nu(\xxo)} < c < 1 $, if $\vnorm{\bggk} \leq \left( {(c - \sqrt{1-\nu(\xxo)}) (1+\beta) [\gamma(\xxo)]^{1+\beta}}/{L_{\HH}} \right)^{1/\beta}$, we get $\vnorm{\bggkk} \leq c \vnorm{\bggk}$.
	In other words, the local convergence rate is \emph{problem-independent}, which is a similar characteristic to that of exact Newton-CG; see \cite{ssn2018}. 
	If $ \nu(\xxo) = \beta = 1 $, then \cref{thm:exact_local}-\ref{cond:thm:exact_local_02} implies that
	\begin{align}
		\label{eq:quadratic_grad}
		\vnorm{\bggkk} \le \frac{L_{\HH}}{2 \gamma(\xxo)^{2}} \vnorm{\bggk}^{2},
	\end{align}
	which greatly resembles the quadratic convergence of Newton-CG, but in terms of $ \vnorm{\bggk} $ in lieu of $ \vnorm{\xxk - \xx^{\star}} $. For strongly-convex objectives, \cref{eq:quadratic_grad} coincides exactly with the well-known bound on $ \vnorm{\bgg} $ in the literature, e.g., see \cite[Eqn.\ (9.33)]{boyd2004convex}. 
	For comparison, the local convergence rate of the Newton-type method proposed in \cite{mishchenko2021regularized} for strongly-convex objectives is superlinear.
\end{remark}

\subsubsection{Inexact Update}
\label{sec:inexact}
Clearly, in almost all practical application, it is rather unreasonable to assume that \cref{eq:least_norm_solution} can be solved exactly. Approximations to \cref{eq:least_norm_solution}, in the similar literature, is typically done by requiring 
\begin{align}
	\label{eq:update_inexact_cg}	
	\vnorm{\HHk \ppk + \bggk} \leq \theta \vnorm{\bggk},
\end{align} 
for some appropriate $ 0 \leq \theta < 1 $. In other words, we can simply require that an approximate solution $ \ppk $ is, at least, better than ``$ \pp = \bm{0} $'' by some factor. Such inexactness conditions have long been used in the analysis of Newton-CG, e.g., see \cite{nocedal2006numerical,byrd2011use,ssn2018,bollapragada2016exact}. However, Newton-MR allows for further relaxation of this condition. Indeed,  as we will later see in this section, we can further loosen \cref{eq:update_inexact_cg} by merely requiring that an approximate solution satisfies
\begin{align}
	\label{eq:update_inexact}
	\dotprod{\HHk \ppk, \bggk } \leq -\frac{(1-\theta)}{2} \vnorm{\bggk}^{2}.
\end{align} 
It is easy to see that \cref{eq:update_inexact_cg} implies \cref{eq:update_inexact}. In our theoretical analysis below, we will employ \cref{eq:pseudo_regularity_lower} and hence we need to ensure that for the update directions, $ \ppk $, we have $ \ppk \in \range (\HHk) $. In this light, we introduce the following relaxation of \cref{eq:update_inexact_cg}
\begin{align}
	\label{eq:update_inexact_range}
	\text{Find} ~~ \pp \in \textnormal{Range}(\HHk) \quad \text{s.t.} \quad  \text{$ \pp  $ satisfies \cref{eq:update_inexact}}.
\end{align}

The inexactness condition in \cref{eq:update_inexact_range} involves two criteria for an approximate solution $ \pp $, namely feasibility of $ \pp $ in \cref{eq:update_inexact} and that $ \pp \in \textnormal{Range}(\HHk) $. To seamlessly enforce the latter, recall that the $ t\th $ iteration of MINRES-QLP can be described as follows \cite[Table 5.1 and Eqn (3.1)]{choi2011minres}
\begin{align}
	\label{eq:minres_qlp}
	\pp^{(t)} = \argmin_{\pp \in \real^{d}} & \; \|\pp\|, \quad \text{subject to} \hspace{0.5cm} \pp \in \Argmin_{\widehat{\pp} \in \mathcal{K}_{t}(\HH, \bgg)} \vnorm{\HH \widehat{\pp} + \bgg},
\end{align}
where $ \mathcal{K}_{t}(\AA, \bb) = \text{Span}\left\{\bb,\AA \bb, \ldots, \AA^{t-1} \bb \right\}$ denotes the $t\th$-Krylov sub-space generated by $ \AA $ and $ \bb$ with $ t \leq r \triangleq \text{rank}\left(\AA\right) $. 
If $ \nu(\xxo) = 1 $ in \cref{assmpt:null_space}, then we necessarily have $ \bgg \in \textnormal{Range}(\HH) $, and hence $ \pp^{(t)} \in \textnormal{Range}(\HH), \forall t $. 
Otherwise, when $ \nu(\xxo) < 1 $ in \cref{assmpt:null_space}, which implies $ \bgg \notin \textnormal{Range}(\HH) $, we cannot necessarily expect to have $ \pp^{(t)} \in \textnormal{Range}(\HH) $. However, one can easily remedy this by modifying MINRES-QLP iterations to incorporate $ \mathcal{K}_{t}(\HH, \HH\bgg) $ instead of $ \mathcal{K}_{t}(\HH, \bgg) $. Compared with \cref{eq:minres_qlp}, this essentially boils down to performing one additional matrix-vector product to compute the vector $ \HH\bgg $, which is then normalized and used as the initial vector within the Lanczos process; see \cite{hanke2017conjugate,calvetti2000curve} for similar modifications applied to MINRES. 
%
Clearly, this will not change the guarantees of MINRES-QLP regarding the monotonicity of the residuals as well as the final solution at termination, i.e., $ \pp^{\dagger} = - \left[\HH \right]^{\dagger} \bgg  $. 
Since for all $ t $, we have $ \pp_{t} \in \textnormal{Range}(\HH) $, it follows that any feasible $ \pp_{t} $ satisfies \cref{eq:update_inexact_range}. 
We note that, in general, the Krylov subspace methods that define their iterates in $\HH \cdot \mathcal{K}_t(\HH, \bgg)$, as opposed to $\mathcal{K}_t(\HH, \bgg)$, can have a slower convergence, e.g., see \cite{estrin2019euclidean}. However, in our experience, such a side-effect can be negligible given the relatively small number of iterations that is often required to satisfy the inexactness condition \cref{eq:update_inexact}, as opposed to \cref{eq:update_inexact_cg}; see also \cref{rem:rel_res}.

For the analysis of Newton-type methods, to the best of our knowledge, \cref{eq:update_inexact_range} has never been considered before and constitutes the most relaxed inexactness condition on the sub-problems in the similar literature. 
\begin{remark}[\cref{eq:update_inexact} is more relaxed than \cref{eq:update_inexact_cg}]
	\label{rem:rel_res}
	For a fixed $ \theta $, the relative residual of any solution to \cref{eq:update_inexact} is often much larger than that required by \cref{eq:update_inexact_cg}. Indeed, from \cref{eq:minres_qlp} and \cite[Lemma 3.3]{choi2011minres}, we have $\dotprod{\pp_{k}, \HHk \left( \HHk \ppk + \bggk \right)} = 0$. Now, from \cref{eq:update_inexact} we get 
	\begin{align*}
		(1+\theta) \vnorm{\bggk}^{2} &\geq 2 \dotprod{\bggk, \HHk \ppk + \bggk} & \text{(From \cref{eq:update_inexact})}\\
		& = 2 \dotprod{\bggk, \HHk \ppk + \bggk} + 2 \dotprod{\HHk \ppk, \left( \HHk \ppk + \bggk \right)} & \text{($ \dotprod{\ppk, \HHk \left( \HHk \ppk + \bggk \right)} = 0 $)}\\
		&= 2 \dotprod{\bggk+\HHk \ppk, \HHk \ppk + \bggk} = 2 \vnorm{\HHk \ppk + \bggk}^{2},
	\end{align*}
	which implies 
	\begin{align}
		\label{eq:rel_res_theta}
		\vnorm{\HHk \ppk + \bggk} &\leq \sqrt{\frac{1+\theta}{2}} \vnorm{\bggk}.
	\end{align}
	Hence, \cref{eq:update_inexact} is equivalent to requiring the relative residual condition albeit with the tolerance of $ \sqrt{(1+\theta)/2} $.
	This in turn implies that for a given $ \theta $, we can satisfy \cref{eq:update_inexact_range} in roughly less than \emph{half} as many iterations as what is needed to ensure the more stringent \cref{eq:update_inexact_cg}.
\end{remark}

The inexactness tolerance has to be chosen with regards to \cref{assmpt:null_space}. Indeed, under \cref{assmpt:null_space}, for the exact solution, $ \ppk = -\left[\HHk\right]^{\dagger} \bggk $, we have $\dotprod{\HHk \ppk, \bggk } = - \langle \HHk \left[\HHk\right]^{\dagger} \bggk, \bggk \rangle \leq -\nu(\xxo) \vnorm{\bggk}^{2}$.
Hence, it suffices to choose $ \theta $ in \cref{eq:update_inexact} such that $ \theta \geq 1-2\nu(\xxo)  $. 
\begin{condition}[Inexactness Tolerance $ \theta $]
	\label{cond:inexact_tolerance}
	The inexactness tolerance in, $ \theta $, in \cref{eq:update_inexact} is chosen such that $\theta \in \left[1-2\nu(\xxo),1\right)$, where $ \nu(\xxo) $ is as in \cref{assmpt:null_space}. 
\end{condition}
Of course, because of $ \nu(\xxo) $, the interval in \cref{cond:inexact_tolerance} also depends on $ \xxo $, which consequently affect the range of possible choices for $ \theta $. In this sense, the proper notation for inexactness tolerance is $ \theta(\xxo) $. However, to simplify our results, we drop the dependence of $ \theta(\xxo) $ on $ \xxo $.  Also, note that if $ \nu(\xxo) > 1/2 $, then we can take $ \theta $ to be negative.

Replacing the exact update \cref{eq:least_norm_solution} in \cref{alg:NewtonMR_Ex} with the inexact condition \cref{eq:update_inexact_range} and considering \cref{cond:inexact_tolerance}, we get an inexact variant of Newton-MR, depicted in \cref{alg:NewtonMR}, and \cref{thm:inexact} provides its convergence properties.
\begin{algorithm}
	\caption{Inexact Newton-MR}
	\begin{algorithmic}
		\vspace{1mm}
		\STATE \textbf{Input:} 
		\begin{itemize}[label=-]
			\item Initial iterate $\xx_{0} \in \real^{d}$, Inexactness tolerance $ \theta $ as in \cref{cond:inexact_tolerance}, Line-search parameter $0 < \rho < 1$, Termination tolerance $0 < \epsilon$
		\end{itemize}
		\vspace{1mm}
		\FOR {$k = 0,1,2, \cdots$ until $ \vnorm{\bggk} \leq \epsilon $} 
		\vspace{1mm}
		\STATE Find $ \ppk $ using \cref{eq:update_inexact_range} with inexactness tolerance $ \theta $ 
		\vspace{1mm}
		\STATE \label{alg:step:alphak:inexact} Find step-size, $\alphak$, such that \eqref{eq:armijo_gen} holds with $ \rho $
		\vspace{1mm}
		\STATE Set $\xxkk  =  \xxk + \alphak \ppk$
		\vspace{1mm}
		\ENDFOR
		\vspace{1mm}
		\STATE \textbf{Output:} $\xx$ for which $ \vnorm{\bgg(\xx)} \leq \epsilon$
	\end{algorithmic}
	\label{alg:NewtonMR}
\end{algorithm}

\begin{theorem}[Convergence of \cref{alg:NewtonMR}]
	\label{thm:inexact}
	Consider \cref{assmpt:diff,assmpt:lipschitz_special,assmpt:pseudo_regularity,assmpt:null_space}. For the iterates of \cref{alg:NewtonMR}, we have 
	\begin{align}
		\label{eq:master_convergence_inexact}
		\vnorm{\bggkk}^2 \le \left( 1 - 2 \rho \widehat{\tau}(\xxo)\vnorm{\bggk}^{(1-\beta)/\beta} \right)\vnorm{\bggk}^2,
	\end{align}
	where $\widehat{\tau}(\xxo) \triangleq \tau(\xxo) \left[ \left(1-\theta\right)/\left( 2 \nu(\xxo)\right) \right]^{{(1+\beta)}/{\beta}}$,	and $ \rho $ is the line-search parameter of \cref{alg:NewtonMR}, $ \nu(\xxo) $ is defined in \cref{assmpt:null_space}, $\tau(\xxo)$ is as defined in \cref{thm:exact} and $ \theta $ is given in \cref{cond:inexact_tolerance}.
\end{theorem}

\begin{proof}
	Similar to the proof of \cref{thm:exact}, we get \cref{eq:grad_upper_bound}. By \cite[Lemma 3.3 and Section 6.6]{choi2014algorithm}, we know that $ \vnorm{\HHk \ppk} $ is monotonically non-decreasing in MINRES-type solvers, i.e., $\vnorm{\HHk \ppk} \leq \vnorm{\HHk \left[\HHk\right]^{\dagger} \bggk} \leq \vnorm{\bggk}$. Since $ \ppk \in \textnormal{Range}(\HHk) $, from \cref{assmpt:pseudo_regularity}  it follows that 
	\begin{align}
		\label{eq:update_upper_bound}
		\vnorm{\ppk} \leq \frac{1}{\gamma(\xxo)} \vnorm{\HHk \ppk} \leq  \frac{1}{\gamma(\xxo) } \vnorm{\bggk}.
	\end{align}
	This, in turn, implies that
	\begin{align*}
		\vnorm{\bggkk}^{2} \leq \vnorm{\bggk}^{2} + 2 \alpha \dotprod{\HHk \bggk, \ppk}  + \frac{2 L(\xx_{0}) \alpha^{\beta+1} }{(\beta+1) [\gamma(\xxo)]^{\beta+1}} \vnorm{\bggk}^{\beta+1}.
	\end{align*}
	Now, as in the the proof of \cref{thm:exact}, to obtain a lower bound on the step-size returned from the line-search, using the above and \cref{eq:armijo_gen}, we consider $ \alpha $ satisfying $
	2 L(\xx_{0}) \alpha^{\beta+1} \vnorm{\bggk}^{\beta+1} \leq (\beta+1) [\gamma(\xxo)]^{\beta+1} \alpha (1-\rho) (1-\theta)\vnorm{\bggk}^2$.
	This, in turn, implies that the step-size returned from the line-search \cref{eq:armijo_gen} must be such that $\alpha \geq \left( {(1-\theta)(1-\rho) (1+\beta) [\gamma(\xxo)]^{1+\beta} }/{(2 L(\xx_{0}))} \right)^{1/\beta} \vnorm{\bggk}^{(1-\beta)/\beta}$.
	With this lower-bound on the step-size, we obtain the desired result by noting that $\vnorm{\bggk}^{2} + 2 \rho \alpha \dotprod{\ppk, \HHk \bggk} \leq \left( 1 - \rho \alpha (1-\theta)\right) \vnorm{\bggk}^2$.
\end{proof}

Recalling that $ 1-\theta \leq 2 \nu(\xxo) $ from \cref{cond:inexact_tolerance}, we can obtain a bound similar to \cref{eq:meaningful_rate} with $ \widehat{\tau}(\xxo) $ replacing $ \tau(\xxo) $, for the rate given by \cref{thm:inexact}.

\begin{remark}
	\label{rem:global_linear_inexact}
	For $ \theta = 1 - 2 \nu(\xxo) $, i.e., when sub-problems are solved exactly, \cref{thm:inexact} coincides with \cref{thm:exact}. More generally, for any $ \nu(\xxo) \in (0,1]$ and $ \theta \in [1-2 \nu(\xxo),1)$, we get a rate similar, up to a constant factor of $ (1-\theta)/(2 \nu(\xxo)) $,  to that of the exact update in \cref{thm:exact}, i.e., \emph{the effects of the problem-related quantities such as $ L(\xx_{0}) $ and $ \gamma(\xxo) $ remain the same}. This is in contrast to Newton-CG, for which in order to obtain a rate similar to that of the exact algorithm, one has to solve the linear system to a high enough accuracy, i.e., $ \theta \leq \sqrt{\kappa} $, where $ \kappa $ is the condition number of the problem. Otherwise, the dependence of the convergence rate on the problem-related quantities is significantly worsened; see \cite[Theorem 2]{ssn2018}. 
\end{remark}

Similar local convergence results as in \cref{thm:exact_local} can also be obtained for the case where the update direction, $ \ppk $, is obtained approximately. Note that, again as in \cref{rem:global_linear_inexact}, when $ \theta = 1-2 \nu(\xxo) $, the results of \cref{thm:inexact_local} coincide with those of \cref{thm:exact_local}.

\begin{theorem}[Error Recursion of \cref{alg:NewtonMR} with $ \alphak = 1 $]
	\label{thm:inexact_local}
	Under the same assumptions of \cref{thm:exact_local}, suppose $ \xxk \in \sX_{0} $, and consider one iteration of \cref{alg:NewtonMR} with $ \alphak = 1 $. 
	\begin{enumerate}[label = (\roman*)]
		\item \label{cond:thm:inexact_local_01} If \cref{assmpt:lipschitz_special} holds, then
		\begin{align*}
			\vnorm{\bggkk}^{2} \leq \frac{2 L(\xx_{0})}{(1+\beta)[\gamma(\xxo)]^{1+\beta}} \vnorm{\bggk}^{1+\beta} + \theta \vnorm{\bggk}^{2}.
		\end{align*}
		\item \label{cond:thm:inexact_local_02} If \cref{assmpt:lip_special_03} holds and $ \HH(\xx) $ is continuous, then
		\begin{align*}
			\vnorm{\bggkk} &\leq \frac{L_{\HH}}{(1+\beta) [\gamma(\xxo)]^{1+\beta}} \vnorm{\bggk}^{1+\beta} + \sqrt{\frac{1+\theta}{2}} \vnorm{\bggk}.
		\end{align*}
	\end{enumerate}
	Here, $ L(\xx_{0}), \gamma(\xxo) $ and $ L_{\HH} $ are defined, respectively, in Assumptions \ref{assmpt:lipschitz_special}, \ref{assmpt:pseudo_regularity}, and \ref{assmpt:lip_special_03}, $\beta$ refers to the respective constants of Assumptions  \ref{assmpt:lipschitz_special} and \ref{assmpt:lip_special_03}, and $ \theta $ is an inexactness tolerance chosen according to \cref{cond:inexact_tolerance}.
\end{theorem}
\begin{proof}
	The proof is very similar to that of \cref{thm:exact_local}. (i) This follows immediately from \cref{eq:grad_upper_bound} with $ \alphak = 1 $ and using \cref{eq:update_inexact} coupled with \cref{eq:update_upper_bound}. (ii) We first replace $ \bggk = -\HHk \ppk +\left(\HHk \ppk + \bggk\right) $ in \cref{eq:mvt}. As in the proof \cref{thm:exact_local}-\ref{cond:thm:exact_local_02}, from \cref{eq:rel_res_theta}, we get $\vnorm{\bggkk} \leq L_{\HH} \vnorm{\ppk}^{1+\beta}/(1+\beta) + \sqrt{(1+\theta)/2} \vnorm{\bggk}$, which using \cref{eq:update_upper_bound} gives the result. 
\end{proof}

\begin{remark}
	\label{rem:local_super_linear_inexact}
	Here, as in \cref{rem:local_super_linear}, one can obtain local linear convergence rate that is problem-independent. 
	For example, from \cref{thm:inexact_local}-\ref{cond:thm:inexact_local_02}, it follows that if $ \sqrt{(1+\theta)/2} < c < 1 $ and the gradient is small enough, we get $\vnorm{\bggkk} \leq c \vnorm{\bggk}$;  see \cite[Section 6]{kelley1995iterative} for a similar statement in the context of more general nonlinear equations.
	This implies that local convergence in the norm of the gradient is very fast even with a very crude solution of the sub-problem. 
	We also note that if $ \nu(\xxo)=\beta=1 $, and $ \theta_{k \in \mathcal{O}(\vnorm{\bggk}^{2}) - 1} $, then \cref{thm:inexact_local}-\ref{cond:thm:inexact_local_02} implies a quadratic convergence rate. 
\end{remark}

Finally, we can obtain iteration complexity to find a solution satisfying $ \vnorm{\bgg(\xx)} \leq \epsilon $ for a desired $ \epsilon > 0 $. 
\begin{corollary}[Iteration Complexity of \cref{alg:NewtonMR}]
	\label{corollary:iter_complexity}
	Under the same assumptions as in \cref{thm:inexact}, consider finding an iterate for which $ \vnorm{\bggk} \leq \epsilon $ for a desired $ \epsilon > 0 $. Then, we have the following two case:
	\begin{enumerate}[label = \textbf{(\roman{*})}]
		\smallskip \item With $\beta \ge 1$ in \cref{assmpt:lipschitz_special} then $\displaystyle k \in \bigO{\log({1}/{\epsilon})}$.
		\smallskip \item With $0 < \beta < 1$ in \cref{assmpt:lipschitz_special}, then $\displaystyle k \in \bigO{\epsilon^{{(\beta-1)}/{\beta}} \log({1}/{\epsilon}) }$. 
	\end{enumerate} 
\end{corollary}
\begin{proof}
	Recall that we always have $ \vnorm{\bggk} \leq \vnorm{\bgg^{(k-1)}}  \leq \ldots \leq \vnorm{\bgg_{0}}.$ Dropping the dependence of $ \widehat{\tau}(\xxo) $ on $ \xxo $, we have the following.
	\begin{enumerate}[label = \textbf{(\roman{*})}]
		\item For $ \beta \geq 1 $, from \cref{eq:master_convergence_inexact}, we have $\vnorm{\bggk}^{2} \le \left( 1 - 2 \rho \widehat{\tau}\vnorm{\bgg_{0}}^{(1-\beta)/\beta} \right)^{k} \vnorm{\bgg_{0}}^2$, which, in order to obtain $ \vnorm{\bggk}^{2} \leq \epsilon $, implies that we must have
		\begin{align*}
			k \geq \log\left(\epsilon/\vnorm{\bgg_{0}}\right)/\log\left( 1 - 2 \rho \widehat{\tau}\vnorm{\bgg_{0}}^{(1-\beta)/\beta} \right).
		\end{align*}
		Now, noting that $ -\log(1-1/x) \in \mathcal{O}(1/x) $, we obtain the result.
		\item For $ \beta < 1 $, as long as $ \vnorm{\bggk}^{2} \geq \epsilon $, from \cref{eq:master_convergence_inexact}, we have 
		\begin{align*}
			\vnorm{\bggk}^{2} &\le \left( 1 - 2 \rho \widehat{\tau}\vnorm{\bggk}^{(1-\beta)/\beta} \right)^{k} \vnorm{\bgg_{0}}^2 \le \left( 1 - 2 \rho \widehat{\tau} \epsilon^{(1-\beta)/\beta} \right)^{k} \vnorm{\bgg_{0}}^2.
		\end{align*}
	\end{enumerate}	
	Now, we obtain the result as above. 
\end{proof}
\cref{corollary:iter_complexity} implies that (inexact) Newton-MR is guaranteed to converge for functions, which traditionally have not been considered suitable candidates for Newton-type methods. For example, \cref{alg:NewtonMR} can still be applied for optimization of a twice continuously differentiable objective for which we have $ \beta \ll 1 $ in \cref{assmpt:lipschitz_special}. Such functions can be extremely non-smooth, and as a consequence, they have been labeled, rather inaccurately, as cases where the application of curvature might be useless, e.g., \cite{arjevani2017oracle_finite,arjevani2017oracle_general}. 

\subsubsection{Convergence Under Generalized Polyak-\L{}ojasiewicz Inequality}
\label{sec:gpl}
The set of global optima of an invex function, $ f(\xx) $, is characterized by the zeros of the gradient, $ \bgg(\xx) $. If, in addition, the distance to the optimal set, in terms of iterates and/or their respective function values, is also somehow related to $ \bgg(\xx) $, then one can obtain rates at which the iterates and/or their objective values approach optimality. In this section, we consider an important sub-class of invex problems, which allows us to do that. 

\begin{definition}[Generalized Polyak-\L{}ojasiewicz Inequality]
	\label{def:gpl}
	Let $\mathcal{X} \subseteq \real^d$ be an open set. A differentiable function $f$ on $\mathcal{X}$ is said to satisfy the generalized Polyak-\L{}ojasiewicz (GPL) inequality on $ \mathcal{X} $ if there exist $ 1 < \eta < \infty $ and $ 0 < \mu < \infty $ such that
	\begin{align}
		\label{eq:gpl}
		f(\xx) - \inf_{\xx \in \mathcal{X}} f(\xx) &\leq \left( \frac{1}{\mu} \vnorm{\bgg(\xx)}^{\eta}\right)^{{1}/{(\eta-1)}}, \hspace{1cm} \forall \xx \in \mathcal{X}. 
	\end{align}
	The class of functions satisfying \cref{eq:error_bound_property} is denoted by $ \mathcal{F}^{GPL}_{\eta,\mu} $.
\end{definition}

It is clear that $ \mathcal{F}^{GPL}_{\eta,\mu} \subset \mathcal{F} $ (cf.\ \cref{def:invex}). Most often in the literature, Polyak-\L{}ojasiewicz (PL) inequality is referred to as \cref{eq:gpl} with $ \eta = 2$, e.g.,~\cite{karimi2016linear}, which excludes many functions. The generalized notion of Polyak-\L{}ojasiewicz in \cref{eq:gpl} with any $ 1< \eta <\infty $ encompasses many of such functions. For example, the weakly convex function $ f(x) = x^{4} $ clearly violates the typical PL (i.e., with $ \eta = 2 $), but it indeed satisfies \cref{eq:gpl} with $ \eta = 4 $ and $ \mu = 256 $. In fact, any polynomial function of the form $ f(x) = \sum_{i=1}^{p} a_{i} x^{2 i} $ with $ a_{i} \geq 0, i = 1,\ldots,p $ satisfies \cref{eq:gpl}. 

But ``how large is the class $ \mathcal{F}^{GPL}_{\eta,\mu} $ in \cref{def:gpl}''? We aim to shed light on this question by considering other classes of functions that are, in some sense, equivalent to $ \mathcal{F}^{GPL}_{\eta,\mu} $. In doing so, we draw similarities from various relaxations of strong convexity. More specifically, to alleviate the restrictions imposed by making strong convexity assumptions, several authors have introduced relaxations under which desirable convergence guarantees of various algorithms are maintained; e.g., the \emph{quadratic growth condition} \cite{anitescu2000degenerate}, the \emph{restricted secant inequality} \cite{tseng2009block}, and the \emph{error bounds} \cite{luo1993error}. The relationships among these classes of functions have also been established; see  \cite{zhang2017restricted,karimi2016linear,zhang2015restricted,schopfer2016linear,necoara2018linear}. We now give natural extensions of these conditions to invex functions and show that $ \mathcal{F}^{GPL}_{\eta,\mu} $ is an equivalent class of functions. 

Let us define the set of optimal points of the problem \cref{eq:obj} as
\begin{align}
	\label{eq:optimal_set}
	\mathcal{X}^{\star} \triangleq \left\{ \xx^{\star} \in \real^{d} \mid  f(\xx^{\star}) \leq f(\xx), ~\forall \xx \in \real^{d} \right\}.
\end{align}
Further, assume that $ \mathcal{X}^{\star} $ is non-empty, and denote the optimal value of \cref{eq:obj} by $ f^{\star} $. Recall that when $ f  $ is invex, $ \mathcal{X}^{\star} $ need not be a convex set, but it is clearly closed; see \cite[p.\ 14]{mishra2008invexity}.

\begin{definition}[Generalized Functional Growth Property]
	\label{def:fun_growth_property}
	A differentiable function $f$ is said to satisfy the generalized functional growth (GFG) property  if there exist $ 1 < \eta < \infty $, $ 0 < \mu < \infty $ and a vector-valued function $\bm{\phi}: \real^{d} \times \real^{d} \rightarrow \real^{d}$, such that
	\begin{align}
		f(\xx) - f^\star  \geq \mu \min_{\yy \in \mathcal{X}^{\star}} \vnorm{\bm{\phi}(\yy,\xx)}^{\eta}, \hspace{1cm} \forall \xx \in \real^{d}. \label{eq:fun_growth_property}
	\end{align}
	The class of functions satisfying \cref{eq:fun_growth_property} with the same $\bm{\phi}$ is denoted by $ \mathcal{F}^{GFG}_{\bm{\phi},\eta,\mu} $.
\end{definition}

\begin{definition}[Generalized Restricted Secant Inequality] 
	\label{def:secant_property}
	For a vector-valued function $\bm{\phi}: \real^{d} \times \real^{d} \rightarrow \real^{d}$ define $ \mathcal{Y}_{\phi}^{\star}(\xx) \triangleq \left\{ \yy^{\star} \in \mathcal{X}^{\star} \mid  \vnorm{\bm{\phi}(\yy^{\star},\xx)} \leq \vnorm{\bm{\phi}(\yy,\xx)}, ~\forall \yy \in \mathcal{X}^{\star} \right\}$.
	A differentiable function $f$ is said to satisfy the generalized restricted secant (GRS) inequality if there exist $ 1 < \eta < \infty $, $ 0 < \mu < \infty $ and a mapping $\bm{\phi}: \real^{d} \times \real^{d} \rightarrow \real^{d}$, such that
	\begin{align}
		\min_{\yy \in \mathcal{Y}_{\phi}^{\star}(\xx)} \dotprod{-\bgg(\xx), \bm{\phi}(\yy,\xx)} \geq \mu \min_{\yy \in \mathcal{X}^{\star}} \vnorm{\bm{\phi}(\yy,\xx)}^{\eta}, \hspace{1cm} \forall \xx \in \real^{d}. \label{eq:secant_property}
	\end{align}
	The class of functions satisfying \cref{eq:secant_property} with the same $ \bm{\phi} $ is denoted by $ \mathcal{F}^{GRS}_{\bm{\phi},\eta,\mu} $.
\end{definition}
It is easy to see that when $ f $ is convex and $ \bm{\phi}(\yy,\xx) = \yy - \xx $, we have $ \mathcal{Y}_{\phi}^{\star}(\xx) = \left\{[\xx]_{\mathcal{X}^{\star}}\right\}$, where $ [\xx]_{\mathcal{X}^{\star}} \triangleq \arg\min_{\yy \in \mathcal{X}^{\star}} \vnorm{\yy - \xx} $ is the unique orthogonal projection of $ \xx $ onto the set of optimal solutions $ \mathcal{X}^{\star} $ (which, in this case, is convex). Hence, the generalized condition \cref{eq:secant_property} coincides with, and hence is a generalization of, the usual definition of restricted secant inequality for convex functions with $ \eta = 2 $; see~\cite{zhang2017restricted,karimi2016linear,zhang2015restricted,schopfer2016linear}.

\begin{definition}[Generalized Error Bound Property]
	\label{def:error_bound_property}
	A differentiable function $f$ is said to satisfy the generalized error bound (GEB) property  if there exist $ 1 < \eta < \infty $, $ 0 < \mu < \infty $ and a vector-valued function $\bm{\phi}: \real^{d} \times \real^{d} \rightarrow \real^{d}$, such that
	\begin{align}
		\label{eq:error_bound_property}
		\vnorm{\bgg(\xx)} \geq \mu \min_{\yy \in \mathcal{X}^{\star}} \vnorm{\bm{\phi}(\yy,\xx)}^{\eta-1}, \hspace{1cm} \forall \xx \in \real^{d}. 
	\end{align}
	The class of functions satisfying \cref{eq:error_bound_property} with the same $ \bm{\phi} $ is denoted by $ \mathcal{F}^{GEB}_{\bm{\phi},\eta,\mu} $.
\end{definition}

\cref{lemma:gpl} establishes a loose notion of equivalence among the classes of functions mentioned above, as they relate to invexity. However, when restricted to a particular class of invex functions for a given $ \bm{\phi} $, \cref{lemma:gpl} shows that $ \mathcal{F}^{GPL}_{\eta,\mu} \cap \mathcal{F}_{\bm{\phi}}  $ can indeed be larger. In contrast, for when $ \bm{\phi}(\yy, \xx) = \yy - \xx $, equivalence between these classes has been established in \cite{karimi2016linear}. This is in particular so since, unlike \cite{karimi2016linear}, here we have made no assumptions on the smoothness of $ f $ such as Lipschitz-continuity of its gradient. This is a rather crucial distinction as for our main results in this paper, we make smoothness assumptions that are less strict that those typically found in the literature; see \cref{assmpt:lipschitz_special} in \cref{sec:assumption}. 

\begin{lemma}
	\label{lemma:gpl}
	\hfill
	\begin{enumerate}[label = \textbf{(\roman*)}]
		\item For any $ \bm{\phi} $, we have
		\begin{align*}
			\mathcal{F}^{GFG}_{\bm{\phi},\eta,\mu} \cap \mathcal{F}_{\bm{\phi}} \subseteq \mathcal{F}^{GRS}_{\bm{\phi},\eta,\mu} \cap \mathcal{F}_{\bm{\phi}} \subseteq \mathcal{F}^{GEB}_{\bm{\phi},\eta,\mu}\cap \mathcal{F}_{\bm{\phi}} \subseteq 
			\mathcal{F}^{GPL}_{\eta,\mu} \cap \mathcal{F}_{\bm{\phi}} .
		\end{align*}
		\item There exists a $ \widehat{\bm{\phi}} $, for which
		\begin{align*}		
			\mathcal{F}^{GPL}_{\eta,\mu} \equiv \mathcal{F}^{GPL}_{\eta,\mu} \cap \mathcal{F}_{\widehat{\bm{\phi}}} \subseteq
			\mathcal{F}^{GFG}_{\widehat{\bm{\phi}},\eta,\mu} \cap \mathcal{F}_{\widehat{\bm{\phi}}}.
		\end{align*}
		\item These classes are equivalent in the sense that 
		\begin{align*}		
			\bigcup_{\bm{\phi}} \left\{\mathcal{F}^{GFG}_{\bm{\phi},\eta,\mu} \cap \mathcal{F}_{\bm{\phi}}\right\} \equiv \bigcup_{\bm{\phi}} \left\{\mathcal{F}^{GRS}_{\bm{\phi},\eta,\mu} \cap \mathcal{F}_{\bm{\phi}} \right\} \equiv \bigcup_{\bm{\phi}} \left\{\mathcal{F}^{GEB}_{\bm{\phi},\eta,\mu}\cap \mathcal{F}_{\bm{\phi}} \right\} \equiv
			\mathcal{F}^{GPL}_{\eta,\mu}.
		\end{align*}
	\end{enumerate}
\end{lemma}

\begin{proof}
	\begin{enumerate}[label = \textbf{(\roman*)}]
		\item We start by showing that, for any $ f \in \mathcal{F}_{\bm{\phi}} $, \cref{eq:fun_growth_property} implies \cref{eq:secant_property}. Consider any $ \xx \in \real^{d} $ and $ \yy^{\star} \in \mathcal{Y}_{\phi}^{\star}(\xx) $. By \cref{eq:fun_growth_property} and  invexity as well as noticing that $ \mathcal{Y}_{\phi}^{\star}(\xx) \subseteq \mathcal{X}^{\star} \Rightarrow f(\yy^{\star}) = f^{\star} $, we have $\mu \min_{\yy \in \mathcal{X}^{\star}} \vnorm{\bm{\phi}(\yy,\xx)}^{\eta} \leq f(\xx) - f(\yy^{\star})  \leq -\dotprod{\bm{\phi}(\yy^{\star}, \xx),\bgg(\xx)}$.
		Since the last inequality holds for all $ \yy^\star \in \mathcal{Y}_{\phi}^{\star}(\xx) $, we can minimize the right-hand side over all $ \mathcal{Y}_{\phi}^{\star}(\xx) $ to get
		\begin{align*}
			\mu \min_{\yy \in \mathcal{X}^{\star}} \vnorm{\bm{\phi}(\yy,\xx)}^{\eta} \leq \min_{\yy \in \mathcal{Y}_{\phi}^{\star}(\xx)} \dotprod{-\bm{\phi}(\yy,\xx),\bgg(\xx)},
		\end{align*}
		which is exactly \cref{eq:secant_property}. A simple application of Cauchy-Schwarz inequality on \cref{eq:secant_property} and noting that $ \min_{\yy \in \mathcal{Y}_{\phi}^{\star}(\xx)} \vnorm{\bm{\phi}(\yy,\xx)} = \min_{\yy \in \mathcal{X}^{\star}} \vnorm{\bm{\phi}(\yy,\xx)} $, it follows that $\mu \min_{\yy \in \mathcal{X}^{\star}} \vnorm{\bm{\phi}(\yy,\xx)}^{\eta} \leq \min_{\yy \in \mathcal{X}^{\star}} \vnorm{\bm{\phi}(\yy,\xx)} \vnorm{\bgg(\xx)}$. If $ \xx $ is such that $ \min_{\yy \in \mathcal{X}^{\star}} \vnorm{\bm{\phi}(\yy, \xx)} = 0 $, then the inequality \cref{eq:error_bound_property} trivially holds, otherwise dividing both sides by $ \min_{\yy \in \mathcal{X}^{\star}} \vnorm{\bm{\phi}(\yy,\xx)} $ gives \cref{eq:error_bound_property}.
		
		To get \cref{eq:gpl} from 	\cref{eq:error_bound_property}, note that by invexity, for any $ \yy \in \mathcal{X}^{\star}$, we have $ f(\xx) - f^\star \leq \vnorm{\bm{\phi}(\yy,\xx)} \vnorm{\bgg(\xx)}$, hence $f(\xx) - f^\star \leq \min_{\yy \in \mathcal{X}^{\star}} \vnorm{\bm{\phi}(\yy,\xx)} \vnorm{\bgg(\xx)}$, which using \cref{eq:error_bound_property} gives \cref{eq:gpl}. 
		
		\item We will show that \cref{eq:gpl} implies that there exists a $ \widehat{\bm{\phi}} $ for which \cref{eq:fun_growth_property} holds. Indeed, for any invex function $ f $, we can always define a corresponding $ \widehat{\bm{\phi}} $ as
		\begin{align*}
			\widehat{\bm{\phi}}(\yy,\xx) = \left\{
			\begin{array}{ll}
				0, \quad &\bgg(\xx) = \bm{0}\\ \\
				\left({\left(f(\yy) - f(\xx)\right)}/{\vnorm{\bgg(\xx)}^{2}}\right) \bgg(\xx), \quad &\bgg(\xx) \neq \bm{0}
			\end{array}
			\right..
		\end{align*}
		Now suppose $ f \in \mathcal{F}^{GPL}_{\eta,\mu} $. For any $ \xx \in \mathcal{X}^{\star} $, the inequality \cref{eq:fun_growth_property} trivially holds. Suppose $ \xx \notin \mathcal{X}^{\star} $, which implies $ \bgg(\xx) \neq \bm{0} $. For any $ \yy \in \mathcal{X}^{\star} $, we have
		\begin{align*}
			\vnorm{\widehat{\bm{\phi}}(\yy,\xx)}^{\eta} = \frac{|f(\yy) - f(\xx)|^{\eta}}{\vnorm{\bgg(\xx)}^{\eta}} = \frac{\left(f(\xx) - f^{\star}\right)^{\eta}}{\vnorm{\bgg(\xx)}^{\eta}} \leq \frac{1}{\mu} \left(f(\xx) - f^{\star}\right), 
		\end{align*}
		where the last inequality follows since by \cref{eq:gpl}, we have $\left(f(\xx) - f^{\star}\right)^{(1-\eta)} \geq {\mu}/{\vnorm{\bgg(\xx)}^{\eta}}, \; \forall \xx \in \mathcal{X}$.
		\item The is a simple implication of the first two parts.  
	\end{enumerate}
\end{proof}

If $ f \in \mathcal{F}^{GPL}_{\eta,\mu} $, where $ \mathcal{F}^{GPL}_{\eta,\mu} $ is the class of Generalized Polyak-\L{}ojasiewicz (GPL) functions as defined in \cref{def:gpl}, we immediately have the following convergence guarantee of \cref{alg:NewtonMR} in terms of objective value $ f(\xxk) $. Similar results for \cref{alg:NewtonMR_Ex} can also easily be  obtained.

\begin{theorem}[Convergence of \cref{alg:NewtonMR} under GPL inequality \cref{eq:gpl}]
	\label{thm:global_GPL}
	Under the same assumptions as in \cref{thm:inexact}, if $f(\xx)$ satisfies GPL inequality \cref{eq:gpl}, there are constants $ 0 < C < \infty $, $ 0< \zeta < 1 $ and $ \omega > 0 $, such that for the iterates of \cref{alg:NewtonMR} we have the following.
	\begin{enumerate}[label = \textbf{(\roman*)}]
		\item With $ \beta \geq 1 $ in \cref{assmpt:lipschitz_special}, we get R-linear convergence as 
		\begin{align*}
			f(\xxk) - \inf_{\xx \in \real^{d}} f(\xx) \leq C \zeta^{k},
		\end{align*}
		where $ C = C(\xx_{0},\eta,\mu) $, and $ 0 < \zeta = \zeta(\xx_{0},\beta,\eta,\rho,\widehat{\tau}(\xxo)) < 1 $.
		\item With $ 0 < \beta < 1 $ in \cref{assmpt:lipschitz_special}, we get R-sublinear convergence as  
		\begin{align*}
			f(\xxk) - \inf_{\xx \in \real^{d}} f(\xx) \leq C \left( \frac{1}{k+1} \right)^{\omega},
		\end{align*}
		where $ C = C(\xx_{0},\eta,\mu,\rho,\beta,\widehat{\tau}(\xxo)) $, and $ \omega = \omega(\eta,\beta) > 0 $.
	\end{enumerate}
	Here, $ \rho $ is the line-search parameter of \cref{alg:NewtonMR}, $ \eta $ and $ \mu $ are as in \cref{def:gpl}, and $ \widehat{\tau}(\xxo) $ is as in \cref{thm:inexact}.
\end{theorem}
\begin{proof}
	\begin{enumerate}[label = \textbf{(\roman*)}]
		\item 
		For notational simplicity, we drop the dependence of $ \widehat{\tau}(\xxo) $ on $ \xxo $. Consider $ \beta \geq 1 $. From \cref{eq:gpl,eq:master_convergence_inexact}, we have
		\begin{align*}
			f(\xxk) - \inf_{\xx \in \real^{d}} f(\xx) &\leq \frac{1}{\mu^{1/(\eta-1)}}  \left(\vnorm{\bggk}^{2}\right)^{{\eta}/{\left(2(\eta-1)\right)}} \\
			&\leq \frac{1}{\mu^{1/(\eta-1)}}  \left(\left( 1 - 2 \rho \widehat{\tau}\vnorm{\bgg_{0}}^{(1-\beta)/\beta} \right)^{k} \vnorm{\bgg_{0}}^2 \right)^{{\eta}/{\left(2(\eta-1)\right)}} \\
			&= \underbrace{\frac{\vnorm{\bgg_{0}}^{{\eta}/{(\eta-1)}} }{\mu^{1/(\eta-1)}}}_{C} \left(\underbrace{\left( 1 - 2 \rho \widehat{\tau}\vnorm{\bgg_{0}}^{(1-\beta)/\beta} \right)^{{\eta}/{\left(2(\eta-1)\right)}}}_{\zeta}\right)^{k},
		\end{align*}
		where $ \widehat{\tau} $ is as in \cref{thm:inexact}, and by \cref{eq:meaningful_rate} we also have that $ 0 < \zeta < 1 $.
		\item Now take $ 0 < \beta < 1 $. From \cref{eq:master_convergence_inexact}, we have
		\begin{align*}
			\vnorm{\bggkk}^2 &\leq \left( 1 - 2 \rho \widehat{\tau}\vnorm{\bggk}^{(1-\beta)/\beta} \right) \vnorm{\bggk}^2 = \vnorm{\bggk}^2 - 2 \rho \widehat{\tau} \vnorm{\bggk}^{(1+\beta)/\beta},
		\end{align*}
		which, implies $\vnorm{\bggkk}^2 \leq \vnorm{\bgg_{0}}^2 - 2 \rho \widehat{\tau}  \sum_{j=0}^{k} \vnorm{\bgg^{j}}^{(1+\beta)/\beta}$.	Rearranging the last inequality gives $2 \rho \widehat{\tau}    \sum_{j=0}^{k} \vnorm{\bgg^{j}}^{(1+\beta)/\beta} \leq \vnorm{\bgg_{0}}^2 - \vnorm{\bggkk}^2 \leq \vnorm{\bgg_{0}}^2$. From \cref{eq:grad_descent}, we get $(k+1) 2 \rho \widehat{\tau}    \vnorm{\bgg^{k}}^{(1+\beta)/\beta} \leq \vnorm{\bgg_{0}}^2$, which gives
		\begin{align*}
			\vnorm{\bgg^{k}} \leq \left((2 \rho \widehat{\tau})^{-1} \vnorm{\bgg_{0}}^2 \right)^{\beta/(1+\beta)}  \left(\frac{1}{k+1}\right)^{\beta/(1+\beta)}.
		\end{align*}
		Using \cref{eq:gpl}, we get 
		\begin{align*}
			f(\xxk) - \inf_{\xx \in \real^{d}} f(\xx) &\leq \frac{1}{\mu^{1/(\eta-1)}}  \vnorm{\bggk}^{{\eta}/{\left((\eta-1)\right)}} \\
			&\leq \frac{1}{\mu^{1/(\eta-1)}}  \left( \left((2 \rho \widehat{\tau})^{-1} \vnorm{\bgg_{0}}^2 \right)^{\beta/(1+\beta)}  \left(\frac{1}{k+1}\right)^{\beta/(1+\beta)} \right)^{{\eta}/{\left((\eta-1)\right)}} \\
			&= \underbrace{\frac{1}{\mu^{1/(\eta-1)}}  \left((2 \rho \widehat{\tau})^{-1} \vnorm{\bgg_{0}}^2 \right)^{\frac{\beta \eta}{(\eta-1)(1+\beta)}} }_{C} \left(\frac{1}{k+1}\right)^{\overbrace{\frac{\beta \eta}{(\eta-1)(1+\beta)}}^{\omega}}. 
		\end{align*}
		
	\end{enumerate}
\end{proof}

The following Lemma shows that, under some assumptions, the norm of the gradient at each point can be estimated using the distance of the point to the optimality set $ \mathcal{X}^{\star} $.
\begin{lemma}
	\label{lemma:lip_grad_star}
	Under \cref{assmpt:lipschitz_special,assmpt:pseudo_regularity,assmpt:null_space}, for any $ \xxo \in \real^{d} $, there is exists a constant $ 0 \leq L_{\bgg}(\xxo) < \infty $, such that 
	\begin{align}
		\label{eq:lip_grad_star}
		\vnorm{\bgg(\xx)} \leq L_{\bgg}(\xxo) \min_{\xx^{\star} \in \mathcal{X}^{\star}} \vnorm{\xx - \xx^{\star}}, \quad \forall \xx \in \sX_{0}, 
	\end{align}
	where $ \mathcal{X}^{\star} $ is as in \cref{eq:optimal_set}.
\end{lemma}
\begin{proof}
	By \cref{assmpt:lipschitz_special}, since $ \mathcal{X}^{\star} \subseteq \mathcal{X}_{0} $, we have $\vnorm{\HH(\xx) \bgg(\xx)} = \vnorm{\HH(\xx) \bgg(\xx) - \HH(\xx^{\star}) \bgg(\xx^{\star})} \leq L(\xx_{0}) \vnorm{\xx - \xx^{\star}}$, $\; \forall (\xxs,\xx) \in \sX^{\star} \times \real^{d} $. Since the left-hand side of the above inequality is independent of $ \xx_{0}$, we get $\vnorm{\HH(\xx) \bgg(\xx)} \leq L \vnorm{\xx - \xx^{\star}}, \; \forall (\xxs,\xx) \in \sX^{\star} \times \real^{d}$, where $ L \triangleq \inf_{\xx_{0} \in \real^{d}} L(\xx_{0})$.
	Now, from \cref{assmpt:pseudo_regularity,assmpt:null_space}, it follows that for any $(\xx,\xx^{\star}) \in \sX_{0} \times \mathcal{X}^{\star} $, we have
	\begin{align*}
		\vnorm{\bgg(\xx)} \leq \frac{1}{\nu(\xxo)} \vnorm{\left[\HH(\xx)\right]^{\dagger} \HH(\xx) \bgg(\xx)} \leq \frac{1}{\nu(\xxo) \gamma(\xxo)} \vnorm{\HH(\xx) \bgg(\xx)} \leq \frac{L}{\nu(\xxo) \gamma(\xxo)} \vnorm{\xx - \xx^{\star}}. 
	\end{align*}
	The result follows by taking the minimum of the right-hand side over $ \sX^{\star} $.
\end{proof}

Using \cref{lemma:lip_grad_star}, from \cref{thm:inexact_local}-\ref{cond:thm:inexact_local_02}, with any $ \xx^{\star} \in \mathcal{X}^{\star} $, we get
\begin{align*}
	\vnorm{\bggkk} \le \frac{L_{\HH} \left[L_{\bgg}(\xxo)\right]^{1+\beta}}{(1+\beta) [\gamma(\xxo)]^{1+\beta}} \vnorm{\xxk - \xx^{\star}}^{1+\beta} + \sqrt{\frac{1+\theta}{2}} L_{\bgg}(\xxo) \vnorm{\xxk - \xx^{\star}},
\end{align*}
where $\gamma(\xxo),\nu(\xxo), L_{\HH}, \theta$, and $L_{\bgg}(\xxo)$ are the constants defined, respectively in \cref{eq:pseudo_regularity}, \cref{eq:null_space}, \cref{eq:lip_special_03}, \cref{eq:update_inexact} and \cref{eq:lip_grad_star}.
Now, suppose we have \cref{eq:error_bound_property} with $ \bm{\phi}(\yy,\xx) = \yy - \xx $ (perhaps only locally). It follows that
\begin{align*}
	\min_{\xx^{\star} \in \mathcal{X}^{\star}} \vnorm{\xxkk - \xx^{\star}}^{\eta-1} \le \frac{L_{\HH} \left[L_{\bgg}(\xxo)\right]^{1+\beta}}{(1+\beta) [\gamma(\xxo)]^{1+\beta} \mu} \vnorm{\xxk - \xx^{\star}}^{1+\beta} + \frac{\sqrt{(1+\theta)/2} L_{\bgg}(\xxo)}{\mu} \vnorm{\xxk - \xx^{\star}}.
\end{align*}

We can also obtain a similar inequality using \cref{thm:inexact_local}-\ref{cond:thm:inexact_local_01}. We can gather the above in the following corollary. 
\begin{theorem}
	\label{thm:exact_inexact_local_pl}
	Consider \cref{assmpt:diff,assmpt:pseudo_regularity,assmpt:null_space}. Suppose that we have \cref{eq:error_bound_property} with $ \bm{\phi}(\yy,\xx) = \yy - \xx $. For the iterates of \cref{alg:NewtonMR} using update directions from \cref{eq:update_inexact_range} and with $ \alphak = 1 $, we have 
	\begin{align*}
		\min_{\xx^{\star} \in \mathcal{X}^{\star}} \vnorm{\xxkk - \xx^{\star}}^{c_{0}(\eta-1)} \le c_{1} \min_{\xx^{\star} \in \mathcal{X}^{\star}} \vnorm{\xxk - \xx^{\star}}^{1+\beta} + c_{2} \min_{\xx^{\star} \in \mathcal{X}^{\star}} \vnorm{\xxk - \xx^{\star}}^{c_{0}},
	\end{align*}
	where
	\begin{enumerate}[label = (\roman*)]
		\item \label{cond:exact_inexact_local_p1} if \cref{assmpt:lipschitz_special} holds, then
		\begin{align*}
			c_{0} = 2, ~~ c_{1} = \frac{2 L(\xx_{0}) \left[L_{\bgg}(\xxo)\right]^{1+\beta}}{(1+\beta) [\gamma(\xxo)]^{1+\beta} \mu^{2}}, ~~ c_{2} = \frac{\theta L^{2}_{\bgg}(\xxo)}{\mu^{2}},
		\end{align*}
		\item \label{cond:exact_inexact_local_p2} and if \cref{assmpt:lip_special_03} holds and $ \HH(\xx) $ is continuous, then
		\begin{align*}
			c_{0} = 1, ~~ c_{1} = \frac{L_{\HH} \left[L_{\bgg}(\xxo)\right]^{1+\beta}}{(1+\beta) [\gamma(\xxo)]^{1+\beta} \mu}, ~~ c_{2} = \frac{L_{\bgg}(\xxo)\sqrt{(1+\theta)/2}}{\mu}.
		\end{align*}
	\end{enumerate}
	Here, $ L(\xx_{0}), \gamma(\xxo), L_{\HH}$ and $\theta$ are defined, respectively, in \cref{assmpt:lipschitz_special,assmpt:pseudo_regularity,assmpt:lip_special_03,cond:inexact_tolerance}, $\mathcal{X}^{\star} $ is as in in \cref{eq:optimal_set}, $ L_{\bgg}(\xxo) $ is as in \cref{lemma:lip_grad_star}, $ \mu $ and $ \eta $ are as in \cref{def:error_bound_property}, and $\beta$ refers to the respective constants of \cref{assmpt:lipschitz_special,assmpt:lip_special_03}.
\end{theorem}

\begin{remark}
	\label{rem:local_optimal_set}
	Assuming \cref{eq:error_bound_property} with $ \bm{\phi}(\yy,\xx) = \yy - \xx$ (locally), is weaker than requiring to have isolated (local) minimum. For example, suppose we have $ \nu(\xxo) = 1 $ and $ \eta = 2 $. By setting $ \theta_{k} \in \bigO{\min_{\xx^{\star} \in \mathcal{X}^{\star}} \vnorm{\xxk - \xx^{\star}}^{2 \beta}} -1 $, from \cref{thm:exact_inexact_local_pl}-\ref{cond:exact_inexact_local_p2}, we get super-linear convergence to the \emph{set of optimal solutions}, $ \mathcal{X}^{\star} $. In fact, the rate is quadratic for $ \beta = 1 $ as a special case, which matches those from many prior works, e.g., \cite{zhou2006convergence,dembo1982inexact,li2004regularized}. 
	Clearly, convergence in terms of $ \min_{\xx^{\star} \in \mathcal{X}^{\star}} \vnorm{\xxk - \xx^{\star}} $ relaxes the notion of isolated minimum, which is at times used for the local convergence analysis of Newton-CG. Hence, error bound conditions similar to that used in \cref{thm:exact_inexact_local_pl} has been extensively used in similar literature to establish convergence to non-isolated (local) minima \cite{fan2005quadratic,yamashita2001rate,fan2012modified,li2009truncated,zhou2006convergence,zhou2005superlinear,li2004regularized}.
\end{remark}

\section{Numerical Experiments}
\label{sec:experiments}
In this section, we evaluate the empirical performance of inexact Newton-MR (\cref{alg:NewtonMR}) as compared with several other Newton-type optimization methods on some machine learning problems. 
Empirical comparisons to a variety of first-order methods are left for future work.
\paragraph{Optimization Methods.}
In the following performance evaluations, we compare Newton-MR with some widely used optimization methods as listed below; see \cite{nocedal2006numerical} for the details of each of these algorithm.
\begin{itemize}[label = -]
	\item \textit{Newton-CG} with Armijo line-search.
	\item \textit{L-BFGS} using strong Wolfe conditions. 
	\item \textit{Nonlinear-CG} using strong Wolfe conditions. In all of experiments, among the many variants of nonlinear-CG, FR-PR variant \cite[Chapter 5]{nocedal2006numerical} performed better on almost all instances.
	\item \textit{Gauss-Newton} with Armijo line-search.
	\item \textit{Trust-region} with CG-Steihaug sub-problem solver as described in \cite[Algorithms 6.1.1]{conn2000trust} and \cite[Algorithm 7.2]{nocedal2006numerical}.
\end{itemize}

The main hyper-parameters for these methods, used in our simulations, can be found in \cref{table:para}. In all of our experiments, we run each method until the norm of the gradient falls below  1E-10, maximum number of iterations are reached, or the algorithm fails to make progress. The latter case is detected when Newton-CG encounters a negative curvature and its CG inner iterations are terminated, or when, for any method, the maximum number of corresponding line-search has been reached and no step-size has been accepted. These scenarios are depicted by a cross ``$\bm{\times}$'' on all the plots. 

\begin{remark}
	Although Newton-CG is not meant to be used on problems where Hessian can become singular or indefinite, we run plain Newton-CG on our examples without any modifications, e.g., we do not attempt to employ the encountered negative curvature as in \cite{royer2018newton,yao2021inexact}. As a result, we terminate its iterations if CG fails. This is a judicious choice to highlight a significant difference between Newton-MR and Newton-CG. Specifically, this serves to distinguish between cases where the trajectory of Newton-CG remains in regions that are locally strongly-convex and those where it enters areas with high degree of weak-convexity or non-convexity. 
	For example, in \cref{sec:gmm}, more often than not, Newton-CG fails at the very outset with $ \xx_{0} $, whereas Newton-MR makes continuous progress. 
\end{remark}

\begin{table}[htbp]
	\begin{center}
		\scalebox{0.9}{
			\begin{tabular}{|c|c|c|c|c|c|}
				\hline
				\multicolumn{1}{|m{2.5cm}|}{\centering CG/MINRES-QLP/Steihaug inexactness tolerance} & \multicolumn{1}{m{1.5cm}|}{\centering History size of L-BFGS} & \multicolumn{1}{m{2cm}|}{\centering Armijo line-search parameter} & \multicolumn{1}{m{3.5cm}|}{\centering Wolfe curvature condition parameter} &
				\multicolumn{1}{m{2cm}|}{\centering Maximum line-search iterations} & \multicolumn{1}{m{2.5cm}|}{\centering Trust-region parameters}\\ [2ex]
				\hline &&&&& \\ [-2ex]
				0.01 & 20 & $10^{-4}$ & \multicolumn{1}{m{3.5cm}|}{\centering 0.9 (L-BFGS) \\ 0.1 (Nonlinear-CG)} & 1,000 & \cite[(6.1.6)]{conn2000trust}\\ [1ex]
				\hline
		\end{tabular}}
	\end{center}
	\caption{\small Hyper-parameters used in optimization methods. The initial trial step-size for all methods, with the exception of nonlinear-CG is set to $ \alpha = 1 $. For nonlinear-CG, we use the strategy described in \cite[Section 3.5]{nocedal2006numerical}. The initial trust-region size is also set to one. \label{table:para}}
\end{table}

\paragraph{Performance Evaluation Measure.}
In all of our experiments, we plot the objective value vs.\ the total
number of oracle calls of function, gradient and Hessian-vector product. This is so since measuring ``wall-clock'' time can be greatly affected by individual implementation details. In contrast, counting the number of oracle calls, as an implementation and system independent unit of complexity, is most appropriate and fair.
More specifically, after computing each function value, computing the corresponding gradient is equivalent to one additional function evaluation. Our implementations are Hessian-free, i.e., we merely require Hessian-vector products instead of using the explicit Hessian. For this, each Hessian-vector product amounts to two additional function evaluations, as compared with gradient evaluation. The number of such oracle calls per iteration for all the algorithms is given in \cref{table:oracle_calls}. 

\begin{table}[htbp]
	\begin{center}
		\scalebox{0.85}{
			\begin{tabular}{|c|c|c|c|c|c|}
				\hline &&&&& \\ [-2ex]
				Newton-MR & L-BFGS & Newton-CG & Nonlinear-CG & Gauss-Newton & TR CG-Steihaug\\ [1ex]
				\hline &&&&& \\ [-2ex]
				$2 + 2 \times N_s + 2  \times N_l$ & $2 + 2 \times N_l$ & $2 + 2 \times N_s + N_l$ & $ 2+ 2 \times N_l$ & $2 + 2  \times  N_s + N_l$ & $2 + \left( 2  \times  N_s + 1 \right) \times N_l$ \\ [1ex]
				\hline
		\end{tabular}}
	\end{center}
	\caption{\small Complexity measure per iteration for each of the algorithms. $N_s$ and $N_l$ denote, respectively, the total number of iterations for the corresponding inner solver and that resulting from performing the line search or the trust-region radius update. \label{table:oracle_calls}}
\end{table}

\subsection{Softmax Regression}
\label{sec:softmax}

Here, we consider the softmax cross-entropy minimization problem without regularization, which is used in machine learning for multi-class classification applications. More specifically, we have
\begin{align}
	\label{eq:softmax}
	f(\xx) \triangleq \mathcal{L}(\xx_{1},\xx_{2},\ldots,\xx_{C-1})  = \sum_{i=1}^{n} \left(\log \left(1+\sum_{c' = 1 }^{C-1} e^{\dotprod{\aa_{i}, \xx_{c'}}}\right)  - \sum_{c = 1}^{C-1}\mathbf{1}(b_{i} = c) \dotprod{\aa_{i},\xx_{c}}\right),
\end{align}
where $\{\aa_i, b_i\}_{i=1}^{n}$ with $\aa_i \in \real^{p}$, $b_i \in \{0, 1, \dots, C\}$ denote the 
training data, $ C $ is the total number of classes for each input data $ \aa_{i} $ and $\xx = (\xx_{1},\xx_{2},\dots,\xx_{C-1})$. Note that, in this case, we have $ d = (C-1) \times p $. It can be shown that, depending on the data, \cref{eq:softmax} is either strictly convex or merely weakly convex.
In either case, however, by a similar analysis as that in \cref{example:erm}, one can show that $ \nu(\xx) = 1 $. Hence, it follows that $\bgg(\xxk) \in \text{Range}\left(\HH(\xxk)\right)$ and CG iterations within Newton-CG are well-defined. 
The datasets use for our experiments in this section are listed in \cref{table:softmax_data} and the performance of each method is depicted in \cref{fig:softmax_obj}.

\begin{table}[!htbp]
	\centering
	\begin{tabular}{|c|c|c|c|c|c|} 
		\hline
		Name & $ n $ & $ p $ & C & $ d = (C-1) \times p $ & Convexity of \cref{eq:softmax} \\ 
		\hline \hline &&&&& \\ [-2ex]
		\texttt{20 Newsgroups} & 10,142 & 53,975 & 20 & 1,025,525 & Weak \\ [1ex] 
		\hline &&&&& \\ [-2ex]
		\texttt{cifar10} & 50,000  &3,072 & 10 & 27,648 & Strict\\ [1ex] 
		\hline &&&&& \\ [-2ex]
		\texttt{covetype} & 435,759 & 54 & 7 & 324 & Strict\\ [1ex] 
		\hline &&&&& \\ [-2ex]
		\texttt{gisette} & 6,000 & 5,000 & 2 & 5,000 & Strict\\ [1ex] 
		\hline &&&&& \\ [-2ex]
		\texttt{mnist} & 60,000  & 784 & 10 & 7,065 & Strict\\ [1ex] 
		\hline &&&&& \\ [-2ex]
		\texttt{UJIIndoorLoc} & 19,937 & 520 & 5 & 2,080 & Strict\\ [1ex] 
		\hline
	\end{tabular}
	\caption{\small Data sets used for example of \cref{sec:softmax}. All data sets are publicly available from \cite{libsvm,Dua:2017}.\label{table:softmax_data}}
\end{table}

\begin{figure}[!h]\centering
	\begin{minipage}[htbp]{0.32\linewidth}
		\subfigure[\texttt{20 Newsgroups}]
		{\includegraphics[scale=0.36]{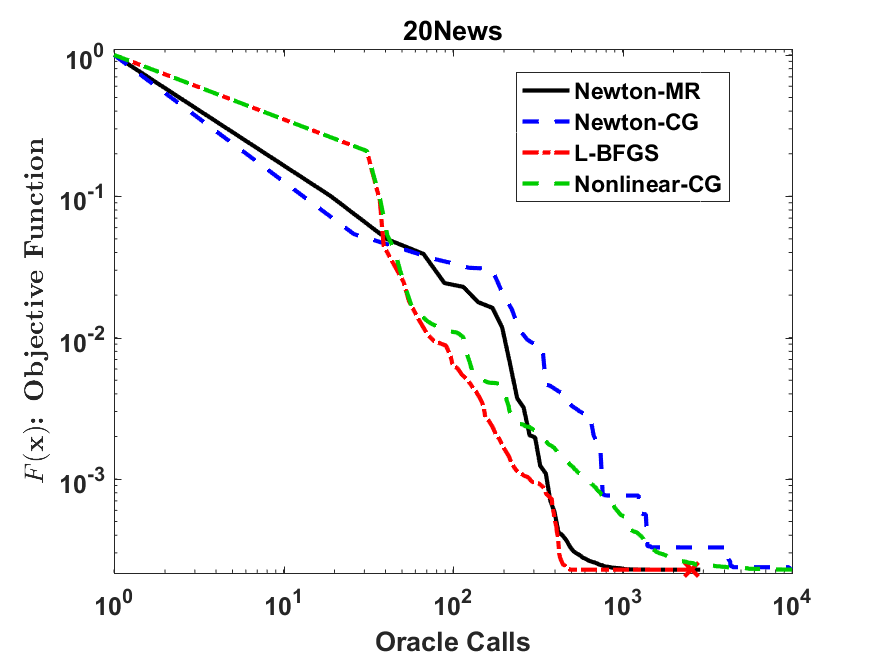}}
	\end{minipage}
	\begin{minipage}[htbp]{0.32\linewidth}
		\subfigure[\texttt{cifar10}]
		{\includegraphics[scale=0.36]{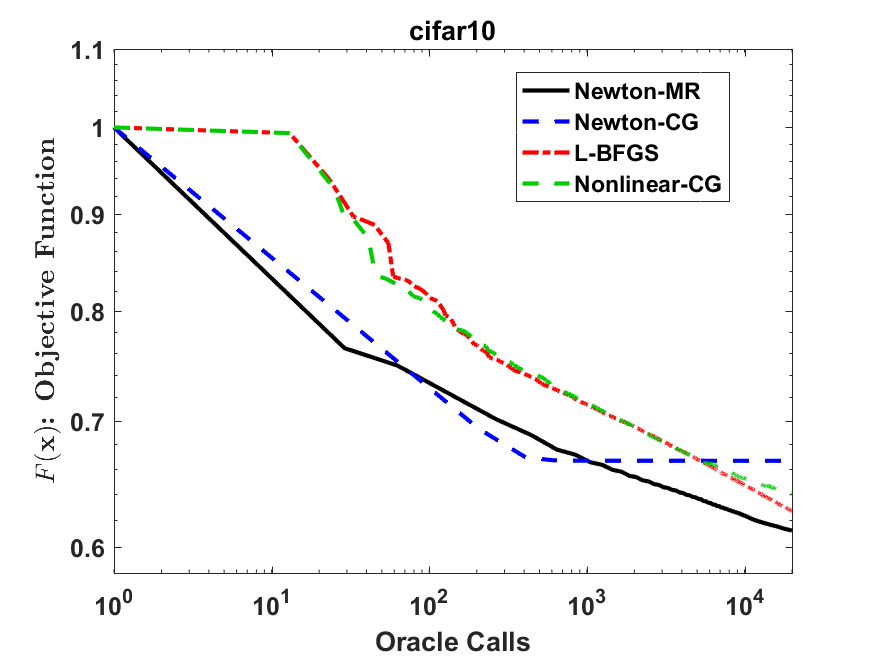}}
	\end{minipage}
	\begin{minipage}[htbp]{0.32\linewidth}
		\subfigure[\texttt{covetype}]
		{\includegraphics[scale=0.36]{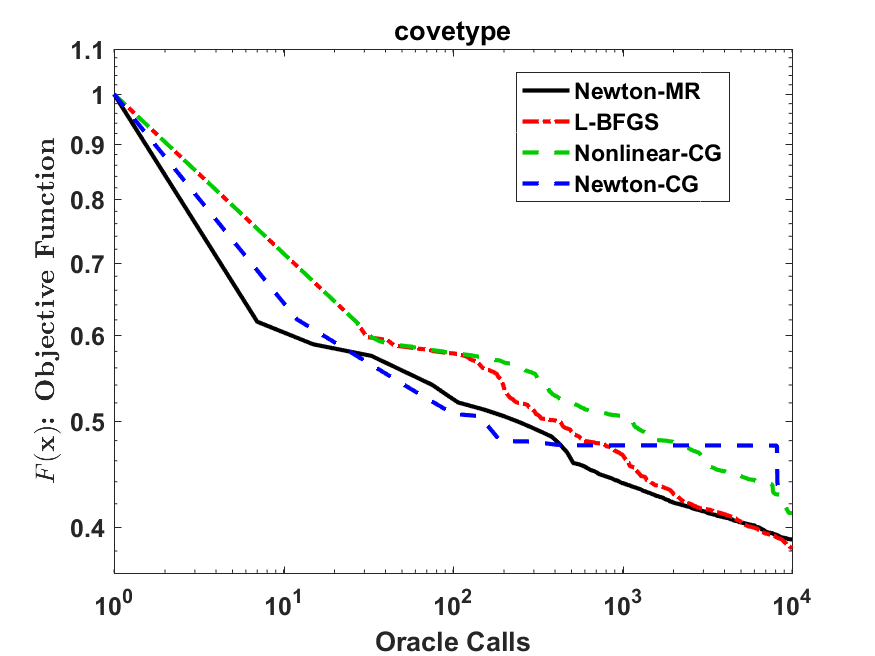}}
	\end{minipage}
	\begin{minipage}[htbp]{0.32\linewidth}
		\subfigure[\texttt{gisette}]
		{\includegraphics[scale=0.36]{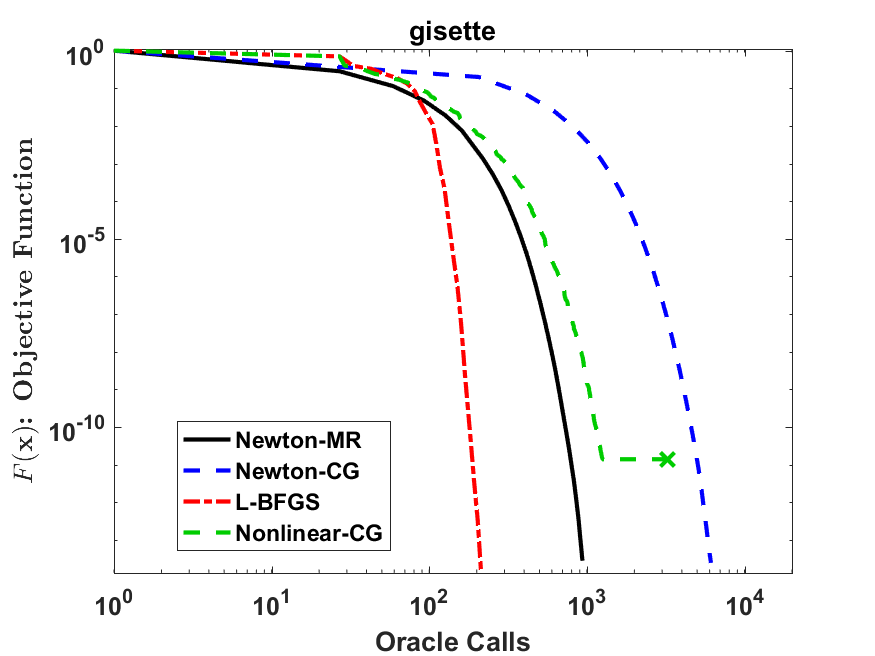}}
	\end{minipage}
	\begin{minipage}[htbp]{0.32\linewidth}
		\subfigure[\texttt{mnist}]
		{\includegraphics[scale=0.36]{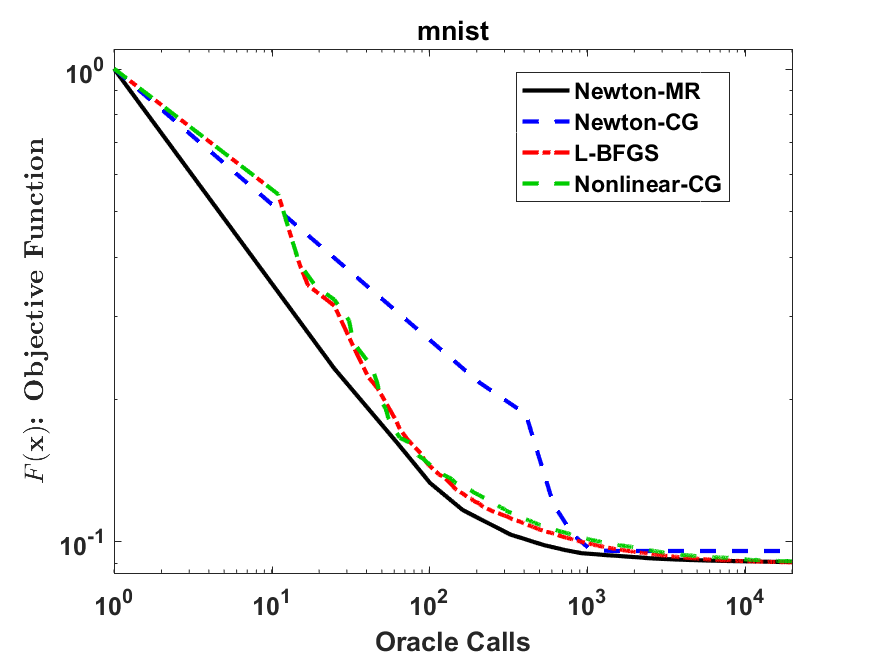}}
	\end{minipage}
	\begin{minipage}[htbp]{0.32\linewidth}
		\subfigure[\texttt{UJIIndoorLoc}]
		{\includegraphics[scale=0.36]{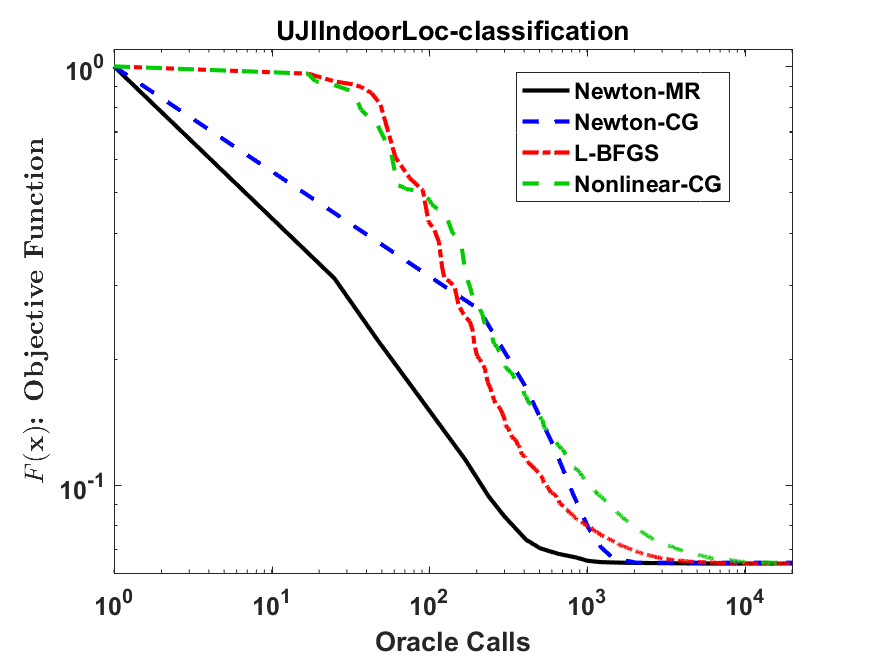}}
	\end{minipage}
	\caption{\small Objective value, $ f(\xx) $ vs.\ number of oracle calls as in \cref{table:oracle_calls} for datasets of \cref{table:softmax_data} on the model problem \cref{eq:softmax}. As for this convex problem, Gauss-Newton is mathematically equivalent to Newton-CG, it is naturally excluded from comparisons. All algorithms are always initialized at $ \xx_{0} = \bm{0} $. As it can be clearly seen, Newton-MR almost always performs very competitively, and in particular, it greatly outperforms Newton-CG. \label{fig:softmax_obj}}
\end{figure}

\subsection{Gaussian Mixture Model}
\label{sec:gmm}
Here, we consider an example involving a mixture of Gaussian densities where the goal is to recover some mixture parameters such as the mean vectors and the mixture weights. Although, this problem is generally not invex, it has been shown in \cite{mei2016landscape} that it indeed exhibits features that are close to being invex, e.g., small regions of saddle points and large regions containing global minimum. Nonetheless, the non-convexity of this problem results in a Hessian matrix that can become indefinite and/or singular across iterations. 

For simplicity, we consider a mixture model with two Gaussian components as 
\begin{align}
	\label{eq:gmm}     
	f(\xx) \triangleq \mathcal{L}(x_0, \xx_1, \xx_2) =  -\sum_{i=1}^n \log\Big(\omega(x_0) \Phi \left(\aa_i; \xx_{1}, \bm{\Sigma}_{1}\right) + (1-\omega(x_0))\Phi \left(\aa_i; \xx_{2},\bm{\Sigma}_{2}\right)\Big), 
\end{align}
where $\Phi$ denotes the density of the p-dimensional standard normal distribution, $\aa_i \in \real^{p}$ are the data points, $\xx_{1} \in \real^{d}, \xx_{2} \in \real^{p}, \bm{\Sigma}_{1} \in \real^{p \times p}$, and $\bm{\Sigma}_{2} \in \real^{p \times p}$ are the corresponding mean vectors and the covariance matrices of two the Gaussian distributions, $x_0 \in \real$, and $\omega(t) = 1/(1+e^{-t}) $ to ensure that the mixing weight lies within $ [0,1] $. Note that, here, $\xx \triangleq [x_{0}, \xx_{1}^{\intercal}, \xx_{2}^{\intercal}]^{\intercal} \in \real^{2p+1}$ and $ d = 2p + 1 $. 

We run the experiments 500 times, and plot the performance profile \cite{dolan2002benchmarking,gould2016note} of each method; see \cref{fig:gmm_profile}. For each run, we generate $1,000$ random data points, generated from the mixture distribution \cref{eq:gmm} with $ p = 100 $, and ground truth parameters as $ x_{0}^{\star} \sim \bm{\mathcal{U}}_{1}[0,1]$, $\xx^{\star}_{1} \sim \bm{\mathcal{U}}_{p}[-1,1],$ and $\xx^{\star}_{2} \sim \bm{\mathcal{U}}_{p}[3,4] $, where $ \bm{\mathcal{U}}_{p}[a,b] $ denotes the distribution of a $p$-dimensional random vector whose independent components are drawn uniformly over the interval $[a,b]$. 
Covariance matrices are constructed randomly, at each run, with controlled condition number, such that they are not axis-aligned. For this, we first randomly generate two $p \times p$ matrices, $\WW_1, \WW_2$, whose elements are iid drawn standard normal distribution. We then find corresponding orthogonal basis, $ \QQ_{1}, \QQ_{2} $, using QR factorizations. This is then followed by forming $ \bm{\Sigma}_{i}^{-1} = \QQ_{i}^{\intercal} \DD \QQ_{i} $ where $ \DD $ is a diagonal matrix whose diagonal entries are chosen equidistantly from the interval $ [0,100] $, i.e., the condition number of each $ \bm{\Sigma}_{i} $ is $ 100 $.

The performance profiles of all the methods are gathered in \cref{fig:gmm_profile}. As it can be seen, Newton-MR greatly outperforms all alternative second-order methods in this example. The trust-region method also shows a reasonable performance. The L-BFGS and Nonlinear-CG methods both have comparable performances. It is worthwhile to highlight that Newton-CG fails to converge in many of the runs. This is  because, for this non-convex problem, the Hessian matrix can become indefinite and, as a result, the CG iterations can fail before the underlying inexactness condition is met. Furthermore, the Gauss-Newton method performed extremely poorly, which can be an indication that the underlying Gauss-Newton matrix is perhaps a low quality approximation to the true Hessian in such a model.

\begin{figure}[!h]\centering
	\subfigure[$ f(\xxk) $]
	{\includegraphics[scale=0.55]{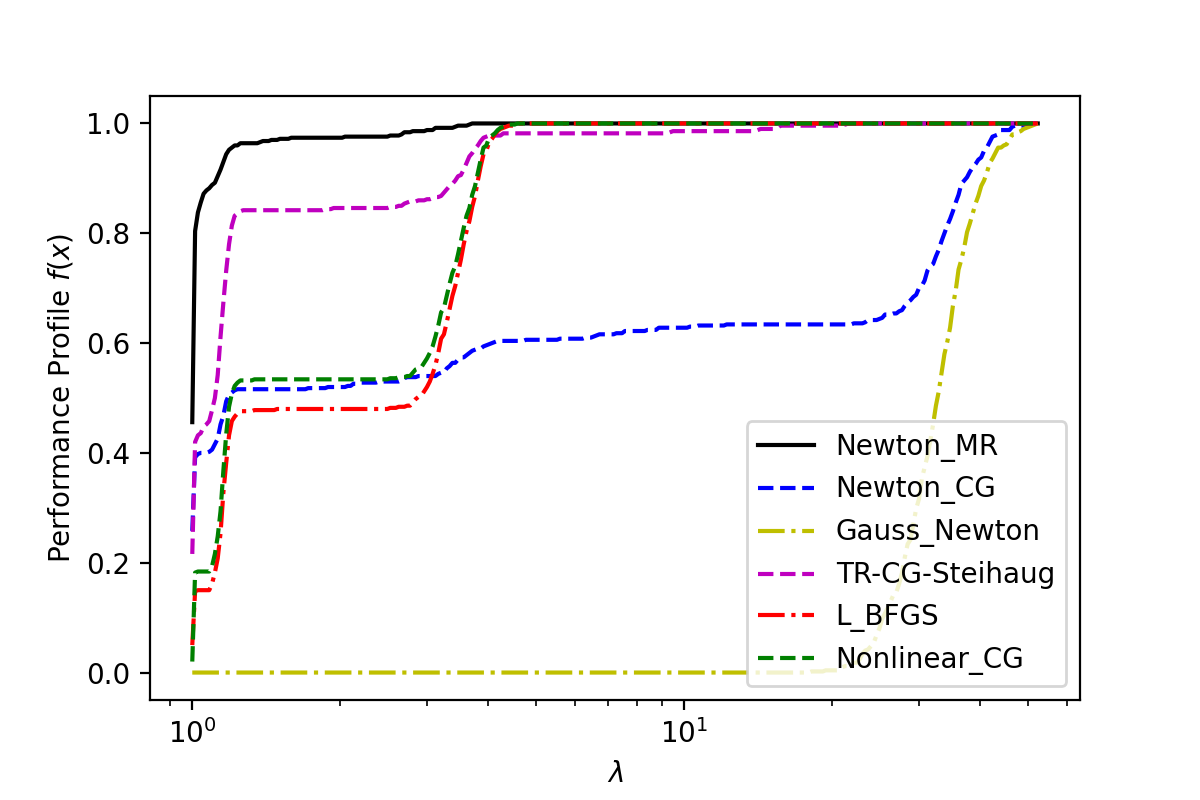}}
	\caption{\small Performance profile for 500 runs of various methods for solving \cref{eq:gmm} as detailed in \cref{sec:gmm}. For a given $ \lambda $ in the x-axis, the corresponding value on the y-axis is the proportion of times that a given solver's performance lies within a factor $\lambda$ of the best possible performance over all runs. All algorithms are always initialized at random using a multivariate normal distribution with mean zero and identity covariance. As it can be seen, Newton-MR, greatly outperforms alternative second-order methods in this example. \label{fig:gmm_profile}}
\end{figure}

\subsection{Nonlinear Least-squares}
\label{sec:nlls}
The examples of \cref{sec:softmax,sec:gmm} show that, when the problem is (approximately) invex, the inexact Newton-MR method, depicted in \cref{alg:NewtonMR}, can be a highly effective optimization algorithm.
We now study a potential drawback of using the plain Newton-MR method for the optimization of the objectives that are highly non-invex. We set out to show that for such functions, \cref{alg:NewtonMR,alg:NewtonMR_Ex}, by virtue of the monotonic reduction of the gradient norm, run the risk of converging to undesirable saddle points or local maxima. For this, following \cite{xuNonconvexEmpirical2017}, we consider the simple, yet illustrative, non-linear least squares problems arising from the task of binary classification with squared loss.\footnote{Logistic loss, the ``standard'' loss used in this task, leads to a convex objective. We use squared loss to obtain a nonconvex objective.} 
More specifically, given the data $\{\aa_i,b_i\}_{i=1}^n$ with $\aa_i\in\bbR^d, b_i\in\{0,1\}$, we consider the regularized objective
\begin{align}
	\label{eq:NLS}	
	f(\xx) = \sum_{i=1}^{n}  \left(b_i - \phi(\dotprod{\aa_i, \xx}) \right)^2 + \psi(\xx),
\end{align}
where $ \phi(z) = 1/(1+e^{-z}) $ is the sigmoid function and $ \psi(\xx) $ is some regularization term. To make the objective highly non-invex, we add a non-convex regularization in the form of $  \psi(\xx) = \sum_{i=1}^{d} x_i^2 / (1+x_i^2) $, where $ x_i $ is the $ i\th $ component of the vector $ \xx $. 

We compare the performance of the inexact variant of Newton-MR, \cref{alg:NewtonMR}, the Gauss-Newton method, as well as the trust-region algorithm with CG-Steihaug solver. The Gauss-Newton method has traditionally been a method of choice for non-linear least-squares problems (of course without a non-convex regularization). The trust-region algorithm can take advantage of the negative curvature directions that arise as part of the Steihaug-CG iterations, and it is an effective method for most-non-convex problems. The approximate Hessian for the Gauss-Newton method is taken as $ \JJ_{\rr}^{\transpose}(\xx) \cdot  \JJ_{\rr}(\xx) + \nabla^{2} \psi(\xx) $, where $ \JJ_{\rr}(\xx) \in \real^{n \times d} $ is the Jacobian of the sigmoid mapping within the squared loss. Since this matrix is not necessarily positive semi-definite (due to non-convexity of $ \psi(\xx) $), the CG inner iterations might fail for this problem in which case we terminate the Gauss-Newton method. We initialize the algorithms using $ \xxo = \zero $, $ \xxo = \one $, and one drawn from the multivariate normal distribution with mean zero and identity covariance. We evaluate the performance of the methods on three datasets from \cref{table:softmax_data}, namely \texttt{20 Newsgroups}, \texttt{mnist} and \texttt{UJIIndoorLoc}. For our binary classification setting, we have relabeled the odd and even classes, respectively, as $ 0 $ and $1 $. The results are depicted in \cref{fig:nnls}.

\begin{figure}[!h]\centering
	\begin{minipage}[htbp]{0.32\linewidth}
		\subfigure[\texttt{20 Newsgroups} ($ \xxo = \zero $)]
		{\includegraphics[scale=0.36]{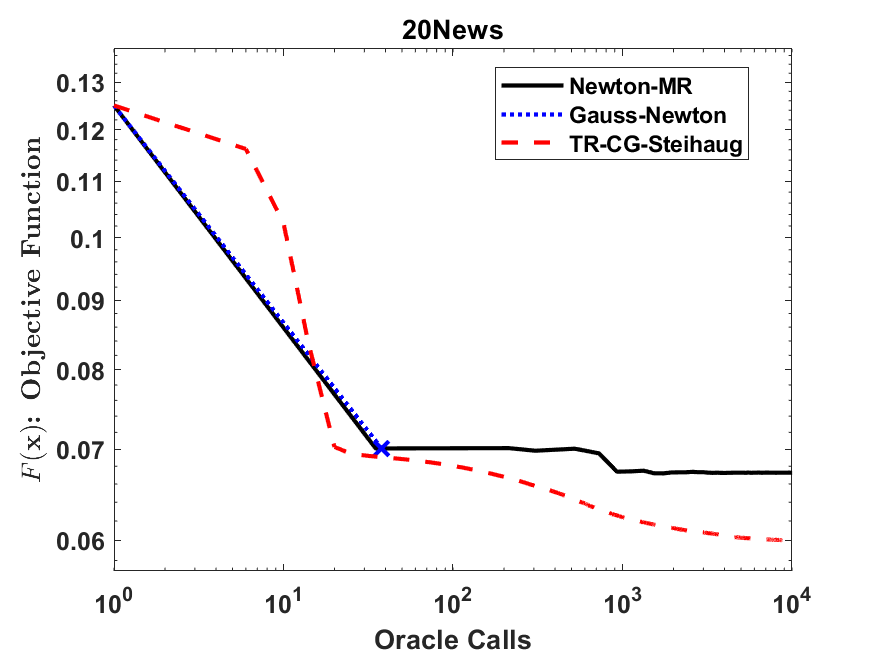}}
	\end{minipage}
	\begin{minipage}[htbp]{0.32\linewidth}
		\subfigure[\texttt{20 Newsgroups} ($ \xxo = \one $)]
		{\includegraphics[scale=0.36]{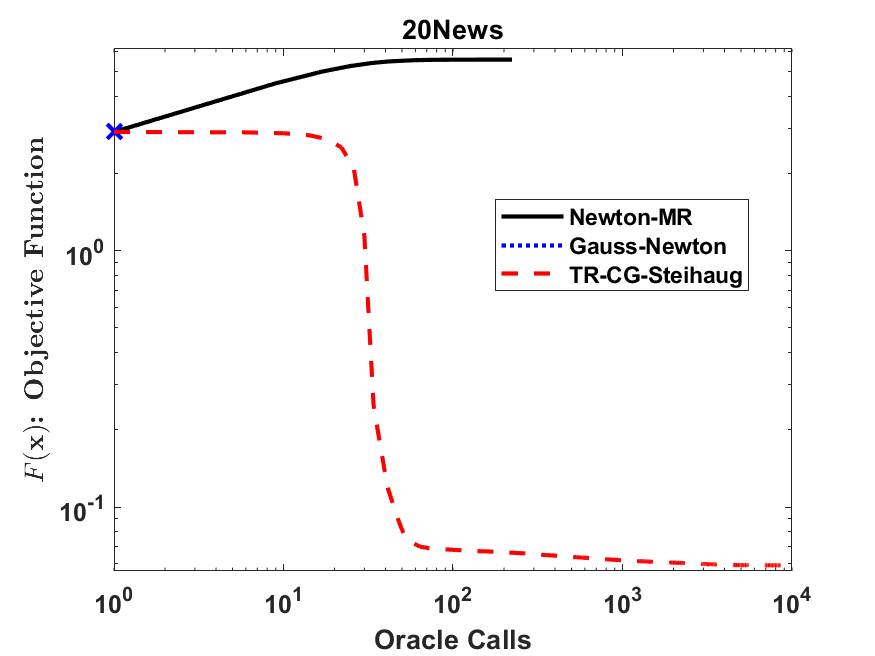}}
	\end{minipage}
	\begin{minipage}[htbp]{0.32\linewidth}
		\subfigure[\texttt{20 Newsgroups} (random)]
		{\includegraphics[scale=0.36]{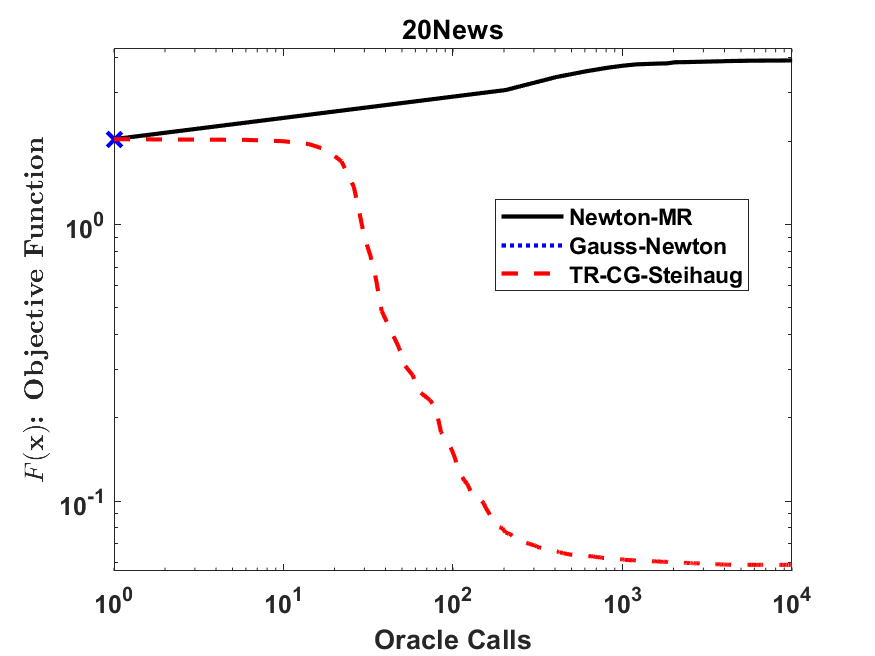}}
	\end{minipage}
	\begin{minipage}[htbp]{0.32\linewidth}
		\subfigure[\texttt{mnist} ($ \xxo = \zero $)]
		{\includegraphics[scale=0.36]{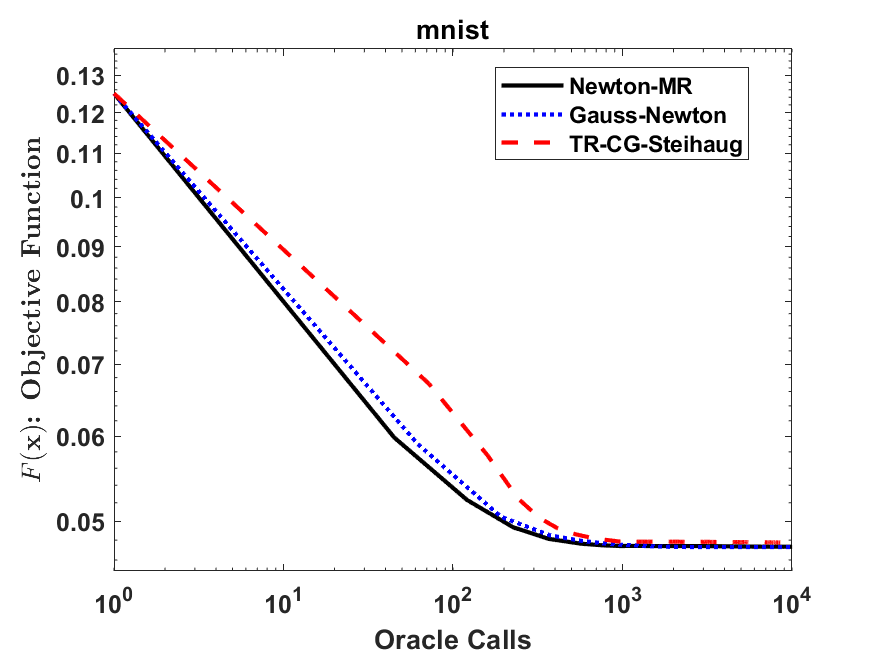}}
	\end{minipage}
	\begin{minipage}[htbp]{0.32\linewidth}
		\subfigure[\texttt{mnist} ($ \xxo = \one $)]
		{\includegraphics[scale=0.36]{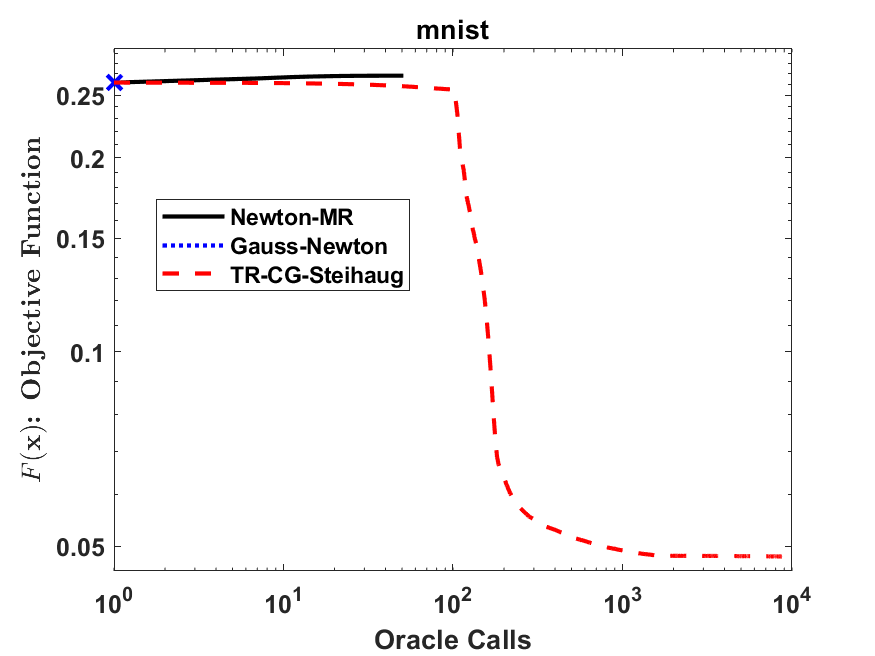}}
	\end{minipage}
	\begin{minipage}[htbp]{0.32\linewidth}
		\subfigure[\texttt{mnist} (random)]
		{\includegraphics[scale=0.36]{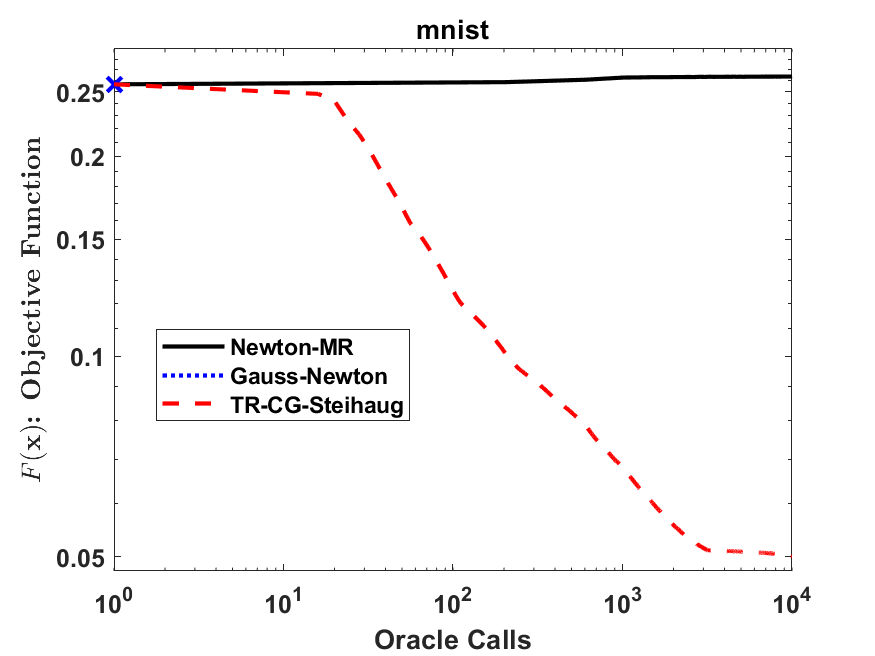}}
	\end{minipage}
	\begin{minipage}[htbp]{0.32\linewidth}
		\subfigure[\texttt{UJIIndoorLoc} ($ \xxo = \zero $)]
		{\includegraphics[scale=0.36]{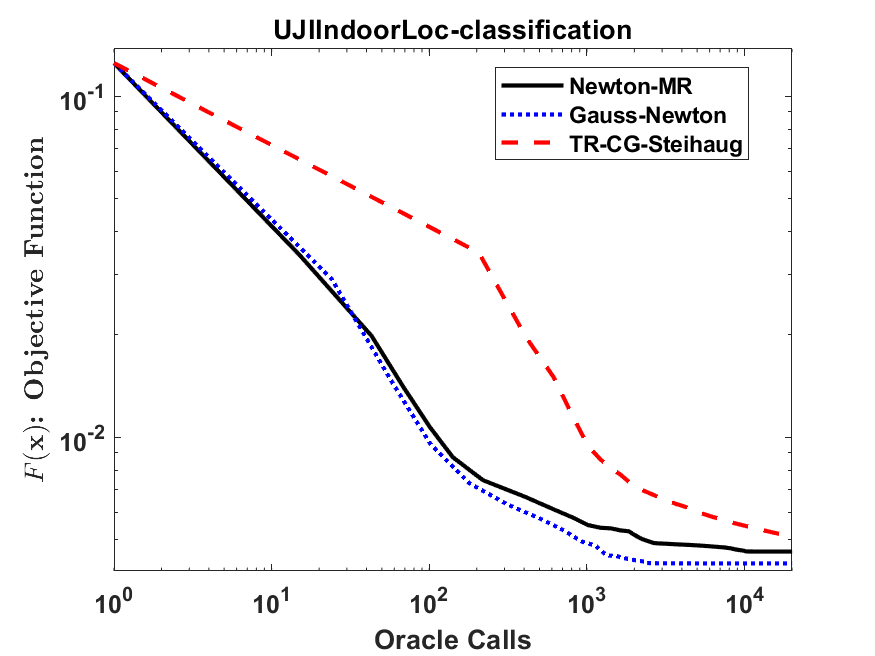}}
	\end{minipage}
	\begin{minipage}[htbp]{0.32\linewidth}
		\subfigure[\texttt{UJIIndoorLoc} ($ \xxo = \one $)]
		{\includegraphics[scale=0.36]{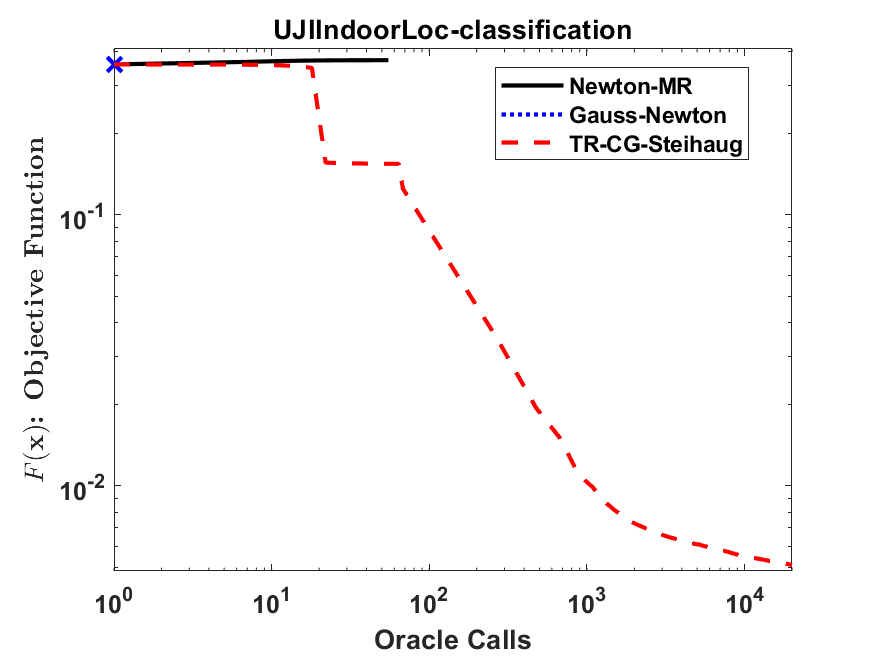}}
	\end{minipage}
	\begin{minipage}[htbp]{0.32\linewidth}
		\subfigure[\texttt{UJIIndoorLoc} (random)]
		{\includegraphics[scale=0.36]{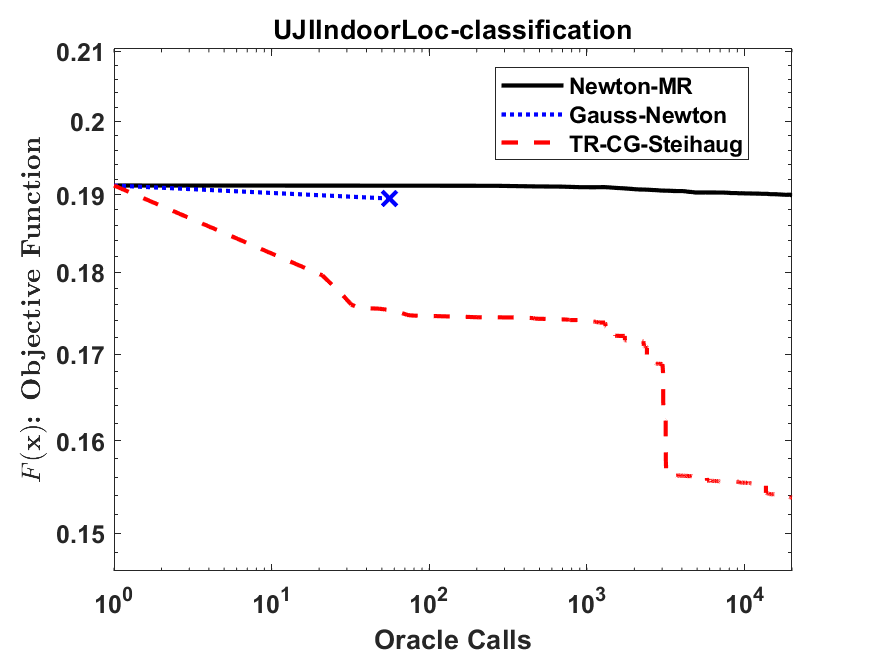}}
	\end{minipage}
	\caption{\small Comparison of the Newton-MR, Gauss-Newton, and trust-region algorithms for a highly non-invex optimization problem, namely the problem \cref{eq:NLS} on three datasets from \cref{table:softmax_data}. We consider three initialization strategies: $ \xx_{0} = \zero $, $ \xxo = \one $, and one that is randomly drawn from the multivariate normal distribution with mean zero and identity covariance. Since this problem is highly non-invex, with unfortunate initialization, Newton-MR can converge to saddle points or even local maxima, whereas the trust-region algorithm, by leveraging the negative curvature direction arising as part of the CG-Steihaug iterations, can navigate its way out of these undesirable regions. The mark ``$ \bm{\times} $'' indicates the Gauss-Newton iteration where the CG method fails due to the indefiniteness of the underlying approximate Hessian matrix.  \label{fig:nnls}}
\end{figure}

It is clear that, when the optimization landscape is somewhat less rugged near the initialization, e.g., the origin on most of the datasets, Newton-MR (as well as Gauss-Newton) can perform reasonably well. However, near more ``hostile'' regions with a high degree of non-convexity, Newton-MR fails to find a good solution and can converge to saddle points or local maxima. In such highly non-convex regions, the matrix used within the Gauss-Newton iterations can also become indefinite, leading to the underlying CG iterations to fail in which case the Gauss-Newton method is terminated (indicated by ``$ \bm{\times} $''). In contrast, by leveraging the negative curvature directions, the trust-region method can reliably converge to a local minima.

Experiments such as those depicted in \cref{fig:nnls} highlight the need for further refinement of the vanilla Newton-MR method studied in this paper to design more general purpose variants, that just like trust-region, can leverage the negative curvature directions that could arise as part of the inner solver. This way, one can extend \cref{alg:NewtonMR_Ex,alg:NewtonMR} far beyond invex problems for arbitrary non-convex objectives. We leave this important venture to future work.

\section{Conclusions}
\label{sec:conclusion}

Motivated to extend the simplicity of the iterations of Newton-CG to beyond strongly convex problems, we consider an alternative algorithm, called Newton-MR, which can be readily used for unconstrained optimization of invex objectives. The iterations of Newton-MR merely involve ordinary least-squares problems, which are (approximately) solved by minimum residual iterative methods, followed by Armijo-type line-search to determine the step-size. We show that under mild assumptions and weak inexactness conditions for sub-problems, we can obtain a variety of local/global convergence results for Newton-MR.

\paragraph{Acknowledgments.}
We are thankful to Prof.\ Sou-Cheng Choi, Prof.\ Michael Saunders and Dr. Ron Estrin for their invaluable help and advice. 
F. Roosta was partially supported by the Australian Research Council through a Discovery Early Career Researcher Award (DE180100923) as well as an Industrial Transformation Training Centre for Information Resilience (IC200100022).

%
%

\appendix
\section{More on \cref{assmpt:null_space}}
This section contains some more examples and insights relating to \cref{assmpt:null_space}.
 
\subsection{\cref{assmpt:null_space} with $ \nu < 1 $}
\cref{example:erm} gives a function that satisfies \cref{assmpt:null_space} with $ \nu = 1 $. We now provide a non-trivial example of a function that satisfies \cref{assmpt:null_space} with $ \nu < 1 $. 
Consider the fractional programming problem of the form 
\begin{align*}
	\min_{(x_{1},x_{2}) \in \mathcal{X}} f(x_{1},x_{2}) = \frac{a x_{1}^2}{b - x_{2}},
\end{align*}
where $ \mathcal{X} = \big\{(x_{1}, x_{2}) \;|\; x_{1} \in \mathbb{R}, \; x_{2} \in (-\infty,b) \cup (b, \infty)\big\} $. The Hessian matrix and the gradient of f can, respectively, be written as
\begin{align*}
	\bgg(\xx) = \left(\begin{array}{c}
		\frac{-2 a x_{1}}{x_{2}-b} \\ \\
		\frac{a x_{1}^{2}}{(x_{2}-b)^{2}}
	\end{array} \right),\quad 	
	\HH(\xx) = \left(\begin{array}{cc}
		\frac{- 2 a }{x_{2} - b} & \frac{2 a x_{1}}{(x_{2}-b)^{2}} \\ \\
		\frac{2 a x_{1}}{(x_{2}-b)^{2}} & \frac{- 2 a x_{1}^{2}}{(x_{2}-b)^{3}}
	\end{array} \right).
\end{align*}
Hence, for a given $\xx$, the range and the null-space of $\HH(\xx)$ are, respectively, given by
\begin{align*}
	\uu_{\xx} = \left(\begin{array}{c}
		-1 \\ \\
		\frac{x_{1}}{x_{2} - b}
	\end{array} \right), \quad 
	\uu^{\perp}_{\xx} = \left(\begin{array}{c}
		\frac{x_{1}}{x_{2} - b} \\ \\
		1
	\end{array} \right).
\end{align*}
Now, we get 
\begin{align*}
	\left|\dotprod{\uu_{\xx},\bgg(\xx)}\right| &= \left|\frac{2 a x_{1}}{x_{2}-b} + \frac{a x_{1}^{3}}{(x_{2}-b)^{3}}\right|, \quad \left|\dotprod{\uu^{\perp}_{\xx},\bgg(\xx)}\right| = \frac{a x_{1}^{2}}{(x_{2}-b)^{2}},
\end{align*}
and
\begin{align*}
	\frac{|\dotprod{\uu_{\xx},\bgg(\xx)}|}{|\dotprod{\uu^{\perp}_{\xx},\bgg(\xx)}|} =  \left|\frac{2 (x_{2}-b)}{x_{1}} + \frac{x_{1}}{(x_{2}-b)}\right| \geq 2 \sqrt{2}, \quad \forall \xx \in \mathcal{X},
\end{align*}
where the lower bound is obtained by considering the minimum of the function $h(y) = {2}/{y} + y$ for $y = {x_{1}}/{(x_{2}-b)}$. Hence, it follows that $\nu = 8/9$. Clearly, $ f $ is unbounded below and, admittedly, this example is of little interest in optimization. Nonetheless, it serves as a non-trivial example of functions that satisfy \cref{assmpt:null_space}.

\subsection{Composition of functions and \cref{assmpt:null_space}}
We now consider \cref{assmpt:null_space}  in the context of the composition of functions. More specifically, consider 
\begin{align}
	\label{eq:comp}
	f(\xx) = h(\rr(\xx)),
\end{align}
where $\rr:\mathbb{R}^{d} \rightarrow \mathbb{R}^{p}$ and $h: \mathbb{R}^{p} \rightarrow \mathbb{R}$ are smooth functions. We have 
\begin{align*}
	\bgg(\xx) \defeq \nabla f(\xx)  &= \JJ^{\transpose}_{\rr} (\xx) \cdot  \nabla h(\rr(\xx)), \\
	\HH(\xx) \defeq \nabla^{2} f(\xx) &= \JJ^{\transpose}_{\rr}(\xx) \cdot  \nabla^2 h(\rr(\xx)) \cdot  \JJ_{\rr}(\xx) + \partial\JJ_{\rr}(\xx) \cdot  \nabla h(\rr(\xx)).
\end{align*}
where $\JJ_{\rr}  (\xx) \in \real^{p \times d}$ and $\partial \JJ_{\rr}(\xx) \in\real^{d\times d\times p}$ are, respectively, the Jacobian and the tensor of all second-order partial derivatives of $ \rr $, and $\nabla h(\rr(\xx)) \in\real^p$, and $\nabla^2 h(\rr(\xx)) \in\real^{p\times p}$ are the gradient and Hessian of $ h $, respectively. \cref{lemma:perturbation} gives a sufficient condition for $ f $ to satisfy \cref{assmpt:null_space}.

\begin{lemma}
	\label{lemma:perturbation}
	Suppose \cref{assmpt:pseudo_regularity} holds, $\nabla^2 h(\rr(\xx))$ is full rank, and also $\textnormal{rank}(\HH(\xx)) \geq \textnormal{rank}(\JJ_{\rr} (\xx))$. If for any $ \xxo \in \real^{d}$, there is some $\nu(\xxo) \in (0,1]$, such that
	\begin{align}
		\label{eq:gn_nu}
		\vnorm{\partial \JJ_{\rr}(\xx) \cdot \nabla h(\rr(\xx))} \le \frac{\gamma(\xxo) \sqrt{1 - \nu(\xxo)}}{2}, \quad \forall \xx \in \mathcal{X}_{0}, 
	\end{align}
	then Assumption \ref{assmpt:null_space} holds with $ \nu(\xxo) $. Here, $\gamma(\xxo) $ is as in \cref{assmpt:pseudo_regularity}.
	
\end{lemma}

\begin{proof}
	Let $ \xxo \in \real^{d} $ be given and take any $ \xx \in \mathcal{X}_{0} $. Define $ \BB \triangleq \AA + \EE $ where
	\begin{align*}
		\BB & = \HH(\xx),\;	\quad \AA =  \JJ^{\transpose}_{\rr}(\xx) \cdot  \nabla^2 h(\rr(\xx)) \cdot \JJ_{\rr}(\xx), \quad	\text{ and } \quad \EE = \partial \JJ_{\rr}(\xx) \cdot \nabla h(\rr(\xx)).
	\end{align*}
	Let $ \UU_{\AA} $ and $ \UU_{\BB} $ be, respectively, arbitrary orthogonal bases for $ \text{Range}(\AA)$ and $\text{Range}(\BB) $ and denote $ \UU^{\perp}_{\AA} $ and $ \UU^{\perp}_{\BB} $ as their respective orthogonal complement. 
	
	By assumption, we have $ r_{\AA} = \text{rank}(\AA) \leq \text{rank}(\BB) = r_{\BB} $. Let $ \widehat{\UU}_{\AA} \in \real^{d \times r_{\BB}}$ be $ \UU_{\AA} $ augmented by any $ r_{\BB} - r_{\AA} $ vectors from $ \UU^{\perp}_{\AA} $ such that $ \text{rank}(\widehat{\UU}_{\AA}) = \text{rank}(\UU_{\BB}) $. By the matrix perturbation theory applied to $\BB = \AA + \EE $, \cite[Theorem 19]{o2018random}, we have
	\begin{align*}
		\vnorm{\widehat{\UU}_{\AA}\widehat{\UU}_{\AA}^{\intercal}  - \UU_{\BB}\UU_{\BB}^{\intercal} } \le \frac{2\vnorm{\EE}}{\sigma_{r_{\BB}}(\BB)} \le \frac{2 \vnorm{ \EE }}{\gamma(\xxo)} \le \sqrt{1 - \nu(\xxo)},
	\end{align*}
	where $ \sigma_{r_{\BB}}(\BB) $ is the smallest non-zero singular value of $ \BB $. This, in turn, implies (see \cite[Theorem 2.5.1]{golub2012matrix})
	\begin{align*}
		\vnorm{\left[\UU^{\perp}_{\BB}\right]^{\intercal} \widehat{\UU}_{\AA}}^2 = \vnorm{\widehat{\UU}_{\AA} \widehat{\UU}_{\AA}^{\intercal}  - \UU_{\BB}\UU_{\BB}^{\intercal} }^2 \le 1 - \nu(\xxo). 
	\end{align*}
	Now since $\nabla^2 h(\rr(\xx))$ is full rank, then $\text{Range}(\AA) = \text{Range}(\JJ^{\transpose}_{\rr}(\xx))$. Therefore $\bgg(\xx) = \UU_{\AA} \UU_{\AA}^{\intercal}  \bgg(\xx)$ and hence $ \UU^{\perp}_{\AA} \left[\UU^{\perp}_{\AA}\right]^{\intercal}  \bgg(\xx) = \bm{0} $, which implies $ \bgg(\xx) = \widehat{\UU}_{\AA} \widehat{\UU}_{\AA}^{\intercal}  \bgg(\xx) $ and $\vnorm{\bgg(\xx)} = \vnorm{\widehat{\UU}_{\AA}^{\intercal}  \bgg(\xx)}$. 
	Now we have
	\begin{align*}
		\vnorm{\left[\UU^{\perp}_{\BB}\right]^{\intercal} \bgg(\xx)}^2 & = \vnorm{\left[\UU^{\perp}_{\BB}\right]^{\intercal}  \widehat{\UU}_{\AA} \widehat{\UU}_{\AA}^{\intercal}  \bgg(\xx)}^2 
		\le \vnorm{\left[\UU^{\perp}_{\BB}\right]^{\intercal}  \widehat{\UU}_{\AA}}^2 \vnorm{\widehat{\UU}_{\AA}^{\intercal} \bgg(\xx)}^2 \\
		& \le \big(1 - \nu(\xxo)\big) \vnorm{\bgg(\xx)}^2 = \big(1 - \nu(\xxo)\big)\left(\vnorm{\UU_{\BB}^{\intercal} \bgg(\xx)}^2 + \vnorm{\left[\UU^{\perp}_{\BB}\right]^{\intercal} \bgg(\xx)}^2\right).
	\end{align*}
	Rearranging the terms above, we obtain
	\begin{align*}
		\vnorm{\left[\UU^{\perp}_{\BB}\right]^{\intercal}  \bgg(\xx)}^2 \le \frac{1 - \nu(\xxo)}{\nu(\xxo)} \vnorm{\UU_{\BB}^{\intercal} \bgg(\xx)}^2,
	\end{align*}
	which completes the proof.  
\end{proof}

\vspace{2mm}
The requirement for $\nabla^2 h(\rr(\xx))$ being full rank is satisfied in many applications, e.g., nonlinear least squares where $ h(\zz) = \vnorm{\zz}^{2}/2 $, statistical parameter estimations using negative log-likelihood where $ h(z) = - \log (z) $, or multi-class classifications using cross-entropy loss where $ h(\zz) = - \sum_{i=1}^{p} y_{i} \log(z_{i})  $ for some $ \sum_{i=1}^{p} y_{i} = 1 $ and $ y_{i} \geq 0 $. The requirement for $\textnormal{rank}(\HH(\xx)) \geq \textnormal{rank}(\JJ_{\rr} (\xx))$, however, might not hold for many problems. Nonetheless, \cref{lemma:perturbation} can give a qualitative guide to understanding the connection between Newton-MR and Gauss-Newton, when both are applied to $f(\xx) = h(\rr(\xx))$. 

\paragraph{Connection to Gauss-Newton.}
Suppose $ f $ in \cref{eq:comp} has isolated local minima, e.g., if $ \JJ_{\rr}(\xxs) $ is full-rank at a local minimum $ \xxs $. Define $ \bSS(\xx) \defeq \partial \JJ_{\rr}(\xx) \cdot \nabla h(\rr(\xx)) $.
It is well-known that the convergence of Gauss-Newton is greatly affected by the magnitude of $ \vnorm{\bSS(\xxs)} $; see \cite[Section 7.2]{sun2006optimization}. Indeed, when $ \vnorm{\bSS(\xxs)} $ is large, Gauss-Newton can perform poorly, whereas small values of $ \vnorm{\bSS(\xxs)} $ implies Gauss-Newton matrix is a good approximation to the full Hessian. In fact, local quadratic convergence of Gauss-Newton can only be guaranteed when $ \vnorm{\bSS(\xxs)} = 0 $; otherwise, the convergence can degrade significantly. For example, if $ \vnorm{\bSS(\xxs)} $ is too large, the Gauss-Newton method may not even converge, while if $ \vnorm{\bSS(\xxs)} $ is small enough relative to $ \vnorm{\JJ^{\transpose}_{\rr}(\xxs) \cdot  \nabla^2 h(\rr(\xxs)) \cdot  \JJ_{\rr}(\xxs)} $, the convergence is linear. 

\cref{lemma:perturbation} allows us to, at least very qualitatively and superficially, compare the performance of Gauss-Newton with Newton-MR for the minimization of \cref{eq:comp}. 
Roughly speaking, \cref{lemma:perturbation} relates $ \vnorm{\bSS(\xxs)} $ to $ \nu $, which directly affects the performance of Newton-MR, e.g.,  see \cref{thm:exact} where larger values of $ \nu $ implies faster convergence for Newton-MR.  
More specifically, suppose $ \vnorm{\bSS(\xxs)} $ is small. Assume that $\bSS(\xxs)$ is smooth enough such that $ \vnorm{\bSS(\xx)} $ is also small for any $ \xx $ in some compact neighborhood of $ \xxs $, denoted by $ \sC $. Let $ \xxo $ be the point in this compact neighborhood such that $ \vnorm{\bgg(\xx)} \leq \vnorm{\bgg(\xxo)}, \; \forall \xx \in \sC $. With $ \sX_{0} $ defined by this $ \xxo $, consider \cref{eq:gn_nu}. Since the left hand side is assumed small, $ \nu(\xxo) $ can be taken large, which directly implies a better convergence rate for Newton-MR. 

In this light, we can expect Newton-MR to perform well whenever Gauss-Newton exhibits good behaviors. 
Our empirical evaluations of \cref{sec:gmm} indeed support this, while hinting at a possibility that the converse might not necessarily be true. 
Nonetheless, to make a fair and more accurate theoretical comparison, one needs to analyze the Gauss-Newton method under the assumptions of this paper (for example by drawing upon the ideas in \cite{eriksson2005regularization} related to the minimum-norm solution). This endeavor is left for future work.

\bibliographystyle{abbrv}
\bibliography{biblio}

\end{document}